\documentclass[a4paper,10pt]{amsart}

\usepackage[T1]{fontenc}
\usepackage[utf8]{inputenc}
\usepackage[english]{babel}
\usepackage{amsmath,amsthm,amssymb,enumerate,stmaryrd}
\usepackage[all]{xy}
\usepackage{hyperref}

\title{Duality for differential operators of Lie-Rinehart algebras}

\date{\today}

\author{Thierry Lambre}
\address[Thierry Lambre]{
  Laboratoire de Math\'ematiques Blaise Pascal, UMR6620 CNRS,
  Université Clermont Auvergne, Campus des Cézeaux, 3 place Vasarely,
  63178 Aubière cedex, France}
\email{thierry.lambre@uca.fr}

\author{Patrick Le Meur}
\address[Patrick Le Meur]{
  Laboratoire de Math\'ematiques Blaise Pascal, UMR6620 CNRS,
  Université Clermont Auvergne, Campus des Cézeaux, 3 place Vasarely,
  63178 Aubière cedex, France}
\curraddr{Universit\'e Paris Diderot, Sorbonne Universit\'e, CNRS, Institut de
   Math\'ematiques de Jussieu-Paris Rive Gauche, IMJ-PRG, F-75013, Paris, France}
\email{patrick.le-meur@imj-prg.fr}

\theoremstyle{plain}
\newtheorem{thm}{Theorem}
\newtheorem{cor}{Corollary}
\newtheorem{lem}{Lemma}[subsection]
\newtheorem{prop}[lem]{Proposition}
\newtheorem{corr}[lem]{Corollary}

\theoremstyle{definition}
\newtheorem{exple}[lem]{Example}
\newtheorem*{rem}{Remark}

\newcommand{\Hom}{\mathrm{Hom}}
\newcommand{\Mod}{\mathrm{Mod}}
\newcommand{\Ext}{\mathrm{Ext}}

\numberwithin{equation}{section}

\subjclass[2010]{Primary 16E35, 16E40, 16S32, 16W25; Secondary 17B63, 17B66}
\keywords{Lie-Rinehart algebra, enveloping algebra, Calabi-Yau
  algebra, skew Calabi-Yau algebra, Van den Bergh duality.}

\begin{document}

\begin{abstract}
  Let $(S,L)$ be a Lie-Rinehart algebra over a commutative ring
  $R$. This article proves that, if $S$ is flat as an $R$-module and
  has Van den Bergh duality in dimension $n$, and if $L$ is finitely
  generated and projective with constant rank $d$ as an $S$-module,
  then the enveloping algebra of $(S,L)$ has Van den Bergh duality in
  dimension $n+d$. When, moreover, $S$ is Calabi-Yau and the $d$-th
  exterior power of $L$ is free over $S$, the article proves that the
  enveloping algebra is skew Calabi-Yau, and it describes a Nakayama
  automorphism of it. These considerations are specialised to Poisson
  enveloping algebras. They are also illustrated on Poisson structures
  over two and three dimensional polynomial algebras and on
  Nambu-Poisson structures on certain two dimensional hypersurfaces.
\end{abstract}

\maketitle

\setcounter{tocdepth}{1}
\tableofcontents

\section*{Introduction}

Rinehart introduced in \cite{MR0154906} the concept of Lie-Rinehart
algebra $(S,L)$ over a commutative ring $R$ and defined its enveloping
algebra $U$. This generalises both constructions of universal
enveloping algebras of $R$-Lie algebras and algebras of differential
operators of commutative $R$-algebras. In \cite{MR1696093},
Huebschmann investigated Poincaré duality on the (co)homology groups
of $(S,L)$. This duality is defined by the existence of a right
$U$-module $C$, called the \emph{dualising} module of $(S,L)$ such
that, for all left $U$-modules $M$ and $k\in \mathbb N$,
\begin{equation}
  \label{eq:66}
\mathrm{Ext}_{U}^k(S,M) \cong \mathrm{Tor}^U_{d-k}(C,M)\,.
\end{equation}
Chemla proved in \cite{MR1718638} that for Lie-Rinehart
algebras arising from affine complex Lie algebroids, the algebra $U$
has a rigid dualising complex, which she determined, and has Van
den Bergh duality \cite{MR1443171}. Having Van den Bergh duality in
dimension $n$ for an $R$-algebra $A$ means that
\begin{itemize}
\item $A$ is homologically smooth, that is $A$ lies in the perfect
  derived category  $\mathrm{per}(A^e)$ of the algebra
  $A^e:=A\otimes_RA^\mathrm{op}$; and
\item $\Ext^\bullet_{A^e}(A,A^e)$ is zero for
  $\bullet\neq 0$ and invertible as an $A$-bimodule if $\bullet=n$.
\end{itemize}
When this occurs, there is a functorial isomorphism, for all
$A$-bimodules $M$ and integers $i$ (see \cite{MR1443171}),
\[
\Ext^i_{A^e}(A,M) \cong \mathrm{Tor}_{n-i}^{A^e}(A,\Ext^n_{A^e}(A,A^e)\otimes_AM)\,;
\]
and $\Ext^n_{A^e}(A,A^e)$ is called the inverse dualising bimodule of
$A$.  Two classes of algebras with Van den Bergh duality are of
particular interest, namely
\begin{itemize}
\item \emph{Calabi-Yau} algebras, for which $\Ext^n_{A^e}(A,A^e)$ is
  required to be isomorphic to $A$ as an $A$-bimodule
  (see~\cite{G06}); and
\item \emph{skew Calabi-Yau} algebras, for which there exists an
  automorphism
  \[
  \nu \in \mathrm{Aut}_{R-\mathrm{alg}}(A)
  \]
  such that
  $\Ext^n_{A^e}(A,A^e)\simeq A^\nu$ as $A$-bimodules (see
  \cite{MR3250287}); here $A^\nu$ denotes the $A$-bimodule obtained
  from $A$ by twisting the action of $A$ on the right by $\nu$.
\end{itemize}
The relevance of these algebras comes from their role in the
noncommutative geometry initiated in \cite{MR917738} and in the
investigation of Calabi-Yau categories, and also from the
specificities of their Hochschild cohomology when $R$ is a field. For
instance, it is proved in \cite{G06,MR2670971} that the Gerstenhaber
bracket of the Hochschild cohomology of Calabi-Yau algebras have a BV
generator.

\medskip

This article investigates when the enveloping algebra $U$ of a general
Lie-Rinehart algebra $(S,L)$ over a commutative ring $R$ has Van den
Bergh duality. 

\medskip

It considers Lie-Rinehart algebras $(S,L)$ such that $S$ has Van den
Bergh duality and is flat as an $R$-module, and $L$ is finitely
generated and projective with constant rank $d$ as an $S$-module.  Under
these conditions, it is proved that $U$ has Van den Bergh
duality. Note that, when $R$ is a perfect field, the former condition
amounts to saying that $S$ is a smooth affine $R$-algebra
\cite{MR2395770}. Note also that, under the latter condition, it is
proved in \cite[Theorem 2.10]{MR1696093} that $(S,L)$ has duality in
the sense of \eqref{eq:66}.  Under the additional assumption that $S$
is Calabi-Yau and $\Lambda^dL$ is free as an $S$-module, it
appears as a corollary that $U$ is skew Calabi-Yau, and a Nakayama
automorphism may be described explicitly. These considerations are
specialised to the situation where the Lie-Rinehart algebra $(S,L)$
arises from a Poisson structure on $S$. Also they are
illustrated by detailed examples in the following cases:
\begin{itemize}
\item For Poisson brackets on polynomial algebras in two or three
  variables; 
\item For Nambu-Poisson structures on two dimensional hypersurfaces
of the shape  $1+T(x,y,z)=0$, where $T$ is a weight homogeneous polynomial.
\end{itemize}

Throughout the article, $R$ denotes a commutative ring, $(S,L)$
denotes a Lie-Rinehart algebra over $R$ and $U$ denotes its enveloping
algebra. Given an $R$-Lie algebra $\mathfrak g$, its universal
enveloping algebra is denoted by $\mathcal U_R(\mathfrak g)$. For an
$R$-algebra $A$, the category of left $A$-modules is denoted by
$\Mod(A)$ and $\Mod(A^\mathrm{op})$ is identified with the category of
right $A$-modules. For simplicity, the piece of notation $\otimes$ is
used for $\otimes_R$. All complexes have differential of degree $+1$.

\section{Main results}
\label{sec:main-results}

A \emph{Lie-Rinehart} algebra over a commutative ring $R$ is a pair
$(S,L)$ where $S$ is a commutative $R$-algebra and $L$ is a Lie
$R$-algebra which is also a left $S$-module, endowed with a homomorphism of
$R$-Lie algebras
  \begin{equation}
    \label{eq:1}
    \begin{array}{rcl}
      L & \to & \mathrm{Der}_R(S) \\
      \alpha & \mapsto & \partial_\alpha:=\alpha(-)
    \end{array}
  \end{equation}
  such that, for all $\alpha,\beta\in L$ and $s\in S$,
  \[
  [\alpha,s\beta] = s[\alpha,\beta] +\alpha(s) \beta\,.
  \]
  Following \cite{MR1696093}, the \emph{enveloping}
  algebra $U$ of $(S,L)$ is identified with the algebra
  \[
  (S\rtimes L)/I\,,
  \]
  where $S\rtimes L$ is the smash-product algebra of $S$ by the action
  of $L$ on $S$ by derivations and $I$ is the two-sided ideal of
  $S\rtimes L$ generated by
  \[
  \{ s \otimes \alpha - 1 \otimes s\alpha\ |\ s\in S,\,\alpha \in L\}
  \]
  (see Lemma~\ref{sec:cont-notat-used}); it is proved in
  \cite{MR1696093} that this set generates $I$ as a right ideal.

  As mentioned in the introduction, when $L$ is a finitely generated
$S$-module with constant rank $d$, the Lie-Rinehart algebra $(S,L)$
has duality in the sense of \eqref{eq:66}  with $C=\Lambda^d_SL^\vee$. Here
$-^\vee$ is the duality $\Hom_S(-,S)$ and $\Lambda^d_SL^\vee$
is considered as a right $U$-module using the Lie derivative
$\lambda_\alpha$, for $\alpha\in L$ (see \cite[Section 2]{MR1696093}),
\[
\lambda_\alpha \colon \Lambda^\bullet_SL^\vee \to
\Lambda^\bullet_SL^\vee\,;
\]
this is the derivation of $\Lambda^\bullet_SL^\vee$ such that, for all
$s\in S$, $\varphi \in L^\vee$ and $\beta \in L$,
\[
\lambda_\alpha(s) = \alpha(s) \ \text{and}\
\lambda_\alpha(\varphi)(\beta) = \alpha(\varphi(\beta)) -
\varphi([\alpha,\beta])\,.
\]
The right $U$-module structure of $\Lambda^d_SL^\vee$ is such
that, for all $\varphi \in \Lambda^d_SL^\vee$ and $\alpha\in L$,
\begin{equation}
  \label{eq:67}
\varphi \cdot \alpha = - \lambda_\alpha(\varphi)\,.
\end{equation}

The first main result of the article gives sufficient conditions for
$U$ to have Van den Bergh duality. It also describes the inverse
dualising bimodule. Here are some explanations on this description.
On one hand, $R$-linear derivations $\partial\in \mathrm{Der}_R(S)$ act
on $\Ext^n_{S^e}(S,S^e)$, $n\in \mathbb N$, by Lie derivatives (see
Section~\ref{sec:left-u-module}),
\[
\mathcal L_\partial \colon \Ext^n_{S^e}(S,S^e) \to \Ext^n_{S^e}(S,S^e)\,.
\]
Combining with the action of $L$ on $S$ yields an action
$\alpha\otimes e \mapsto \alpha \cdot e$ of $L$ on
$\Ext^n_{S^e}(S,S^e)$ such that, for all $\alpha\in L$ and
$e\in \Ext^n_{S^e}(S,S^e)$,
\[
\alpha \cdot e = \mathcal L_{\partial_\alpha}(e)\,,
\]
Although this is not a $U$-module structure on $\Ext^n_{S^e}(S,S^e)$,
it defines a left $U$-module structure on
$\Lambda^d_SL^\vee \otimes_S \Ext^n_{S^e}(S,S^e)$, $d\in \mathbb N$,
such that, for all $\alpha \in L$, $\varphi \in \Lambda^d_SL^\vee$ and
$e\in \Ext^n_{S^e}(S,S^e)$,
\[
\alpha \cdot (\varphi \otimes e) = -\varphi \cdot \alpha \otimes e +
\varphi \otimes 
\alpha \cdot e\,.
\]
On the other hand, consider the functor
\[
F\colon \Mod(U) \to \Mod(U^e)
\]
(see Section~\ref{sec:functor-fcolon-modu}) such that, if
$N\in \Mod(U)$, then $F(N)$ equals $U\otimes_S N$ in $\Mod(U)$ and has
a right $U$-module structure defined by the following formula, for all
$\alpha\in L$, $u\in U$ and $n \in N$,
\[
( u \otimes n) \cdot \alpha = u\alpha \otimes n - u \otimes \alpha
\cdot n\,.
\]
This functor takes left $U$-modules which are invertible as
$S$-modules to invertible $U$-bimodules (see
Section~\ref{sec:invert-u-bimod-1}).  The main result of this article
is the following.
\begin{thm}
  \label{sec:introduction}
  Let $R$ be a commutative ring. Let $(S,L)$ be a Lie-Rinehart algebra
  over $R$. Denote by $U$ the enveloping algebra of $(S,L)$. Assume that
  \begin{itemize}
  \item $S$ is flat as an $R$-module,
  \item $S$ has Van den Bergh duality in dimension $n$,
  \item $L$ is finitely generated and projective with constant rank
    $d$ as an $S$-module.
  \end{itemize}
  Then, $U$ has Van den Bergh duality in dimension $n+d$ and there is
  an isomorphism of $U$-bimodules
  \[
  \Ext^{n+d}_{U^e}(U,U^e) \simeq F(\Lambda_S^dL^\vee \otimes_S \Ext^n_{S^e}(S,S^e))\,.
  \]
\end{thm}
Note that when $R$ is Noetherian and $S$ is finitely generated as an
$R$-algebra and projective as an $R$-module, then there is an
isomorphism of $S$-(bi)modules
\[
\Ext^n_{S^e}(S,S^e) \simeq \Lambda_S^n\mathrm{Der}_R(S)\,;
\]
this isomorphism is compatible with the actions by Lie derivatives
(see Section~\ref{sec:part-case-calabi}). The above theorem was proved in
\cite[Theorem 4.4.1]{MR1718638} when $R=\mathbb C$ and $S$ is finitely
generated as a $\mathbb C$-algebra.

The preceding theorem specialises to the situation where the involved
invertible $S$-modules are free. On one hand, when $(\Lambda^d_SL)^\vee$ is
free as an $S$-module with free generator $\varphi_L$, there is
an associated \emph{trace} mapping
\[
\lambda_L \colon L \to S
\]
such that, for all $\alpha \in L$,
\[
\varphi_L \cdot \alpha = \lambda_L(\alpha)\cdot \varphi_L\,,
\]
where the action on the left-hand side is given by \eqref{eq:67} and
that on the right-hand side is just given by the $S$-module structure.
On the other hand, when $S$ is Calabi-Yau in dimension $n$, each
generator of the free of rank one $S$-module $\Ext^n_{S^e}(S,S^e)$
determines a volume form $\omega_S\in \Lambda^n_S\Omega_{S/R}$, and
the \emph{divergence}
\[
\mathrm{div}\colon \mathrm{Der}_R(S) \to S
\]
associated with $\omega_S$ is defined by the following equality, for
all $\partial\in \mathrm{Der}_R(S)$
\[
\mathcal L_\partial(\omega_S) = \mathrm{div}(\partial)\omega_S\,.
\]
(see Section~\ref{sec:part-case-calabi} for details).
The second main result of the article then reads as follows.
\begin{thm}
  \label{sec:introduction-1}
  Let $R$ be a commutative ring. Let $(S,L)$ be a Lie-Rinehart algebra
  over $R$. Denote by $U$ the enveloping algebra of $(S,L)$. Assume that
  \begin{itemize}
  \item $S$ is flat as an $R$-module,
  \item $S$ is Calabi-Yau in dimension $n$,
  \item $L$ is finitely generated and projective with constant rank
    $d$ and $\Lambda^d_SL$ is free as $S$-modules.
  \end{itemize}
  Then, $U$ is skew Calabi-Yau with a
  Nakayama automorphism $\nu \in \mathrm{Aut}_R(U)$
  such that, for all $s\in S$ and $\alpha\in L$,
  \[
  \left\{
    \begin{array}{rcl}
      \nu(s) & = & s\\
      \nu(\alpha) & = & \alpha + \lambda_L(\alpha) + \mathrm{div}(\partial_\alpha),
    \end{array}\right.
  \]
  where $\lambda_L$ is any trace mapping on $\Lambda^d_SL^\vee$ and $\mathrm{div}$ is any divergence.
\end{thm}

Among all Lie-Rinehart algebras, those arising from Poisson structures
on $S$ play a special role because of the connection to Poisson
(co)homology. Recall that any $R$-bilinear Poisson bracket $\{-,-\}$
on $S$ defines a Lie-Rinehart algebra structure on
$(S,L) = (S,\Omega_{S/R})$ such that, for all $s,t\in S$,
\begin{itemize}
\item $\partial_{ds} = \{s,-\}$;
\item $[ds,dt] = d\{s,t\}$.
\end{itemize}
In this case, the formulations of Theorems~\ref{sec:introduction} and
\ref{sec:introduction-1} simplify because, when $\Omega_{S/R}$ is
projective with constant rank $n$ as an $S$-module, the right
$U$-module structure of $\Lambda^n_S\Omega_{S/R}^\vee$ (see
\eqref{eq:67}) is given by classical Lie derivatives, that is, for all
$s\in S$,
\begin{equation}
  \label{eq:41}
  \lambda_{ds}(\varphi) = \mathcal L_{\{s,-\}}(\varphi)\,.
\end{equation}
More precisely, these theorems specialise as follows.
\begin{cor}
  \label{sec:main-results-1}
  Let $R$ be a Noetherian ring. Let $(S,\{-,-\})$ be a finitely
  generated Poisson algebra over $R$. Denote by $U$ the enveloping
  algebra of the associated Lie Rinehart algebra
  $(S,\Omega_{S/R})$. Assume that
  \begin{itemize}
  \item $S$ is projective in $\Mod(R)$;
  \item $S\in \mathrm{per}(S^e)$;
  \item $\Omega_{S/R}$, which is then projective in $\Mod(S)$, has
    constant rank $n$.
  \end{itemize}
  Then, $U$ has Van den Bergh duality in dimension $2n$ and there is
  an isomorphism of $U$-bimodules
  \[
  \Ext^{2n}_{U^e}(U,U^e) \simeq U \otimes_S \Lambda^n_S\mathrm{Der}_R(S)
  \otimes_S \Lambda^n_S\mathrm{Der}_R(S)\,,
  \]
  where the right-hand side term is a left $U$-module in a natural way
  and a right $U$-module such that, for all $u\in U$, $\varphi,\varphi'\in
  \Lambda^n_S\mathrm{Der}_R(S)$ and $s\in S$,
  \[
  (u \otimes \varphi \otimes \varphi')\cdot ds = u\,ds \otimes \varphi \otimes
  \varphi'
  -  u\otimes (\mathcal L_{\{s,-\}}(\varphi)\otimes \varphi'+ \varphi
  \otimes \mathcal L_{\{s,-\}}(\varphi'))\,.
  \]
  In particular, if $S$ has a volume form, then $U$ is skew Calabi-Yau
  with a Nakayama automorphism $\nu \colon U \to U$ such that, for all
  $s\in S$,
  \[
  \left\{
    \begin{array}{rcl}
      \nu(s) & = & s \\
      \nu(ds) & = & ds + 2\,\mathrm{div}(\{s,-\})\,,
    \end{array}
    \right.
  \]
  where $\mathrm{div}$ is the divergence of the chosen volume form.
\end{cor}
For the case where $R=\mathbb C$ and $S$ is finitely generated as a
$\mathbb C$-algebra, the above corollary is announced in \cite[Theorem
0.7, Corollary 0.8]{MR3687261} using the main results of
\cite{MR1718638}.

This article is structured as
follows. Section~\ref{sec:poincare-duality-s-1} presents useful
information on the case where $S$ has Van den Bergh
duality. Section~\ref{sec:material-u-bimodules} is devoted to
technical lemmas on $U$-(bi)modules. In particular, it presents the
above mentioned functor $F$ and its right adjoint $G$, which play an
essential role in the proof of the main
results. Section~\ref{sec:left-u-module} introduces the action of $L$
on $\Ext^\bullet_{S^e}(S,S^e)$ by Lie derivatives. This structure is
used in Section~\ref{sec:proof-main-theorem} in order to describe
$\Ext^\bullet_{U^e}(U,U^e)$ and prove Theorem~\ref{sec:introduction},
Theorem~\ref{sec:introduction-1} and
Corollary~\ref{sec:main-results-1}. Finally, Section~\ref{sec:example}
applies this corollary to a class of examples of Nambu-Poisson
surfaces.

\section{Poincaré duality for \texorpdfstring{$S$}{S}}
\label{sec:poincare-duality-s-1}

As proved in \cite{MR1443171} when $R$ is field, if $S$ has Van den
Bergh duality in dimension $n$, then there is a functorial
isomorphism, for all $S$-bimodules $N$,
\[
\Ext^\bullet_{S^e}(S,N)\simeq \mathrm{Tor}_{n-\bullet}^{S^e}(S,\Ext^n_{S^e}(S,S^e)
\otimes _S N)\,.
\]
It is direct to check that this is still the case
without assuming that $R$ is a field. In view of the proof of the main
results of the article, \ref{sec:fund-class-contr} relates the
above mentioned isomorphism to the fundamental class of $S$,
following \cite{MR2670971}, and \ref{sec:poincare-duality-s} relates
Van den Bergh duality to the regularity of commutative algebras,
following \cite{MR2395770}.

\subsection{Fundamental class and contraction}
\label{sec:fund-class-contr}

Consider a projective resolution $P^\bullet$ in $\Mod(S^e)$,
\[
\cdots \to P^{-2} \to P^{-1} \to P^0 \xrightarrow{\epsilon} S\,,
\]
and let $p^0\in P^0$ be such that $\epsilon(p_0)=1_S$. For all
$M,N\in \Mod(S^e)$ and $n\in \mathbb N$, define the contraction
\[
\begin{array}{ccc}
  \mathrm{Tor}_n^{S^e}(S,M) \times \Ext^n_{S^e}(S,N) & \to & \mathrm{Tor}_0^{S^e}(S,M\otimes_SN) \\
  (\omega,e) & \mapsto & \iota_e(\omega)
\end{array}
\]
as the mapping induced by the following one
\[
\begin{array}{ccc}
  M \otimes_{S^e}P^{-n}
  & \to &
          \Hom_R(\Hom_{S^e}(P^{-n},N) , (M\otimes_S
          N)\otimes_{S^e}P^0) \\
  x \otimes p & \mapsto & 
                          (\varphi \mapsto (x\otimes \varphi(p))
                          \otimes p^0)\,.
\end{array}
\]
This makes sense because $P^\bullet$ is concentrated in nonpositive
degrees. The construction depends neither on the choice of $p^0$ nor
on that of $P^\bullet$.

Following the proof of \cite[Proposition 3.3]{MR2670971}, when $S\in
\mathrm{per}(S^e)$ and $n$ is taken equal to $\mathrm{pd}_{S^e}(S)$, the
contraction induces an isomorphism for all $N\in \Mod(S^e)$,
\[
\begin{array}{rrcl}
  \mathrm{Tor}_n^{S^e}(S,\Ext^n_{S^e}(S,S^e)) & \to & \Hom_{S^e}(\Ext^n_{S^e}(S,N), \mathrm{Tor}_0^{S^e}(S,\Ext^n_{S^e}(S,S^e)\otimes_SN)) \\
  \omega & \mapsto & \iota_?(\omega)
\end{array}
\]
In the particular case $N=S^e$, the \emph{fundamental class} of $S$ is
the element $c_S\in \mathrm{Tor}_n^{S^e}(S,\Ext^n_{S^e}(S,S^e))$ such
that
\[
(\iota_?(c_S))_{|\Ext^n_{S^e}(S,S^e)} = \mathrm{Id}_{\Ext^n_{S^e}(S,S^e)}\,.
\]

Following the arguments in the proof of \cite[Théorème
4.2]{MR2670971}, when $S$ has Van den Bergh duality in dimension
$n$, which gives that $n=\mathrm{pd}_{S^e}(S)$, the contraction with
$c_S$ induces an isomorphism, for all $N\in \Mod(S^e)$,
\begin{equation}
  \label{eq:30}
  \iota_?(c_S) \colon \Ext^n_{S^e}(S,N) \xrightarrow{\sim} \mathrm{Tor}_0^{S^e}(S,\Ext^n_{S^e}(S,S^e)\otimes_SN)\,.
\end{equation}

When $S$ is projective in $\Mod(R)$, the Hochschild complex
$S^{\otimes \bullet+2}$ is a resolution of $S$ in $\Mod(S^e)$ and the
contraction
\[
\begin{array}{ccc}
  \mathrm{Tor}_n^{S^e}(S,M) \times \Ext^m_{S^e}(S,N) & \to & \mathrm{Tor}_{n-m}^{S^e}(S,M\otimes_SN) \\
  (\omega,e) & \mapsto & \iota_e(\omega)
\end{array}
\]
may be defined for all $M,N\in \Mod(S^e)$ and $m,n\in \mathbb N$, as
the mapping induced at the level of Hochschild (co)chains by
\[
\begin{array}{ccc}
  M \otimes S^{\otimes n} \times \Hom_R(S^{\otimes m},N)
  & \to &
          (M\otimes_S N) \otimes S^{\otimes (m-n)} \\
  ((x|s_1|\cdots|s_n),\psi) & \mapsto & (x \otimes
                                        \psi(s_1|\cdots|s_m)|s_{m+1}|\cdots|s_n)\,. 
\end{array}
\]
When, in addition, $S$ has Van den Bergh duality in dimension $n$,
then \cite[Théorème 4.2]{MR2670971} asserts that the following mapping
given by contraction with $c_S$ is an isomorphism, for all
$N\in \Mod(S^e)$ and $m\in \mathbb N$,
\[
\iota_?(c_S) \colon \Ext^m_{S^e}(S,N) \to \mathrm{Tor}_{n-m}^{S^e}(S,\Ext^n_{S^e}(S,S^e) \otimes_S N)\,.
\]

\subsection{Relationship to regularity}
\label{sec:poincare-duality-s}

The main results of this article assume that $S$ has Van den Bergh
duality. For commutative algebras, this property is related to
smoothness and regularity. The relationship is detailed in
\cite{MR2395770} for the case where $R$ is a perfect field and is
summarised below in the present setting.
\begin{prop}[{\cite{MR2395770}}]
  \label{sec:relat-regul}
  Let $R$ be a Noetherian commutative ring. Let $S$ be a finitely
  generated commutative $R$-algebra and projective as an
  $R$-module. Let $n \in \mathbb N$. The following properties are
  equivalent.
  \begin{enumerate}[(i)]
  \item $S$ has Van den Bergh duality in dimension $n$.
  \item $\mathrm{gl.dim}(S^e)<\infty$ and $\Omega_{S/R}$, which is then
    projective in $\Mod(S)$, has constant rank $n$.
  \end{enumerate}
  When these properties are true, $\mathrm{gl.dim}(S)<\infty$ and
  $\mathrm{Ext}^n_{S^e}(S,S^e)\simeq \Lambda^n_S\mathrm{Der}_R(S)$ as $S$-modules.
\end{prop}
\begin{proof} See \cite{MR2395770} for full details.
  Since $S$ is projective over $R$, then
  $\mathrm{pd}_{(S^e)^e}(S^e)\leqslant 2\,\mathrm{pd}_{S^e}(S)$
  (\cite[Chap. IX, Proposition 7.4]{MR0077480}); besides, using the
  Hochschild resolution of $S$ in $\Mod(S^e)$ yields that
  \[
  \mathrm{gl.dim}(S)\leqslant \mathrm{pd}_{S^e}(S) \leqslant
  \mathrm{gl.dim}(S^e)\,;
  \]
   thus
  \begin{equation}
    \label{eq:57}
    \begin{array}{rcl}
      S \in \mathrm{per}(S^e) & \Leftrightarrow & \mathrm{gl.dim}(S^e)<\infty \\
                           & \Rightarrow & \mathrm{gl.dim}(S) <\infty\,.
    \end{array}
  \end{equation}
  Note also that, following \cite[Theorem 3.1]{MR0142598},
  \begin{equation}
    \label{eq:58}
    \mathrm{gl.dim}(S^e)<\infty \Rightarrow \text{$\Omega_{S/R}$ is projective
      in $\Mod(S)$.}
  \end{equation}
  
  Denote by $\mu$ the multiplication mapping $S\otimes S \to S$.
  Assume $\mathrm{gl.dim}(S^e)<\infty$, let
  $\mathfrak p \in \mathrm{Spec}(S)$ ($\subseteq \mathrm{Spec}(S^e)$)
  and denote by $d$ the rank of $(\Omega_{S/R})_{\mathfrak p}$. Since
  $\Omega_{S/R}\simeq \mathrm{Ker}(\mu)/\mathrm{Ker}(\mu)^2$ as
  modules over $S$ ($\simeq S^e/\mathrm{Ker}(\mu)$) and
  $\mathrm{gl.dim}(S^e)<\infty$, the $(S^e)_{\mathfrak p}$-module
  $\mathrm{Ker}(\mu)_{\mathfrak p}$ is generated by a regular sequence
  having $d$ elements. There results a Koszul resolution of
  $S_{\mathfrak p}$ in $\Mod((S^e)_{\mathfrak p})$. Using this
  resolution and the isomorphism
  $\Ext^\bullet_{S^e}(S,S^e)_{\mathfrak p} \simeq
  \Ext_{(S^e)_{\mathfrak p}}^\bullet (S_{\mathfrak p},
  (S^e)_{\mathfrak p})$
  in $\Mod((S^e)_{\mathfrak p})$ yields isomorphisms of
  $(S^e)_{\mathfrak p}$-modules
  \begin{equation}
    \label{eq:59}
    \Ext^\bullet_{S^e}(S,S^e)_{\mathfrak p} \simeq
    \left\{
      \begin{array}{ll}
        0 & \text{if $\bullet \neq d$} \\
        S_{\mathfrak p} & \text{if $\bullet = d$.}
      \end{array}
    \right.
  \end{equation}

  \medskip

  Now assume $(i)$. Then, $\mathrm{gl.dim}(S^e)<\infty$ (see
  (\ref{eq:57})), $\Omega_{S/R}$ is projective (see \ref{eq:58}) and
  has constant rank $n$ (see (\ref{eq:59})). Conversely, assume that
  $\mathrm{gl.dim}(S^e)<\infty$ and $\Omega_{S/R}$ has constant rank
  $n$. Then, $S\in \mathrm{per}(S^e)$ (see (\ref{eq:57})) and the
  $S$-module (equivalently, the symmetric $S$-bimodule)
  $\Ext^\bullet_{S^e}(S,S^e)$ is zero if $\bullet \neq n$ and is
  invertible if $\bullet = n$ (see (\ref{eq:59})). Thus,
  \[
  (i) \Leftrightarrow (ii)\,.
  \]

  \medskip

  Finally, assume that both $(i)$ and $(ii)$ are true. Then, $\mathrm{gl.dim}(S) <\infty$ (see (\ref{eq:57})). Moreover, Van den Bergh
  duality \cite[Theorem 1]{MR1443171} does apply here and provides an
  isomorphism of $S$-modules
  \[
  \Ext^0_{S^e}(S,\Ext^n_{S^e}(S;S^e)^{-1}) \simeq \mathrm{Tor}_n^{S^e}(S,S)\,,
  \]
  whereas \cite[Theorem 3.1]{MR0142598} yields an isomorphism of
  $S$-modules
  \[
  \mathrm{Tor}_n^{S^e}(S,S) \simeq \Lambda^n_S\Omega_{S/R}\,.
  \]
  Thus, $\Ext^n_{S^e}(S,S^e) \simeq \Lambda^n_S\mathrm{Der}_R(S)$ in $\Mod(S)$.
\end{proof}

\section{Material on \texorpdfstring{$U$}{U}-(bi)modules}
\label{sec:material-u-bimodules}

The purpose of this section is to introduce an adjoint pair of
functors $(F,G)$ between $\Mod(U)$ and $\Mod(U^e)$. In the proof of
Theorem~\ref{sec:introduction}, the $U$-bimodule
$\Ext^\bullet_{U^e}(U,U^e)$ is described as the image under $F$ of a
certain left $U$-module which is invertible as an $S$-module. This
section develops the needed properties of $F$. Hence,
\ref{sec:basic-constr-u} recalls the basic constructions of
$U$-modules; \ref{sec:funct-gcol-modu-1} and
\ref{sec:functor-fcolon-modu} introduce the functors $G$ and $F$,
respectively; \ref{sec:adjunction-between-f} proves that $(F,G)$ is
adjoint; \ref{sec:10} introduces and collects basic properties of
compatible left $S\rtimes L$-modules, these are applied in
Section~\ref{sec:left-u-module} to the action of $L$ on
$\Ext^\bullet_{S^e}(S,S^e)$ by Lie derivatives; and
\ref{sec:invert-u-bimod-1} proves that the functor $F$ transforms left
$U$-modules that are invertible as $S$-modules into invertible
$U$-bimodules. These results are based on the description of $U$ as a
quotient of the smash-product $S\rtimes L$ given in the following
lemma. This description is established in \cite[Proposition
2.10]{MR3659329} in the case of Lie-Rinehart algebras arising from
Poisson algebras.
\setcounter{lem}{0}
\begin{lem}
  \label{sec:cont-notat-used}
  The algebra $S\rtimes L$ has the following properties.
  \begin{enumerate}
  \item The identity mappings $\mathrm{Id}_S\colon S\to S$ and $\mathrm{Id}_L\colon L \to L$ induce an isomorphism of $R$-algebras
    \begin{equation}
      \label{eq:2}
      (S\rtimes L) / I \to U\,,
    \end{equation}
    where $I$ is the two-sided ideal of the smash-product algebra $S
    \rtimes L$ generated by
    \[
    \{s \otimes \alpha - 1 \otimes s\alpha\ |\ s\in S\,,\alpha\in
    L\}\,.
    \]
  \item If $L$ is projective as a left $S$-module, then $U$ is
    projective both as a left and as a right $S$-module.
  \end{enumerate}
\end{lem}
\begin{proof}
  (1) Recall (see \cite{MR0154906}) that $U$ is defined as follows:
  Endow $S\oplus L$ with an $R$-Lie algebra structure such that, for
  all $s,t\in S$ and $\alpha,\beta\in L$,
\[
[ s + \alpha, t + \beta ] = \alpha(t) - \beta(s) +[ \alpha , \beta ]\,;
\]
then, $U$ is the factor $R$-algebra of the subalgebra of the universal
enveloping algebra $\mathcal U_R(S\oplus L)$ generated by the image of
$S \oplus L$ by the two-sided ideal generated by the classes in
$\mathcal U_R( S \oplus L)$ of the following elements, for $s,t \in
S$ and $\alpha \in L$,
\[
s\otimes t - st\,,\ s \otimes \alpha - s\alpha\,.
\]
Recall also that $S\rtimes L$ is the $R$-algebra with underlying
  $R$-module
  \[
  S\otimes \mathcal U_R(L)\,,
  \]
  such that the images of $S\otimes 1$ and $1\otimes \mathcal U_R(L)$ are
  subalgebras, and the following hold, for all $s,t\in S$ and
  $\alpha,\beta\in S$,
  \[
  \left\{
    \begin{array}{l}
      (s\otimes 1) \cdot (1\otimes \alpha) = s\otimes \alpha \\
      (1\otimes \alpha) \cdot (s\otimes 1) = \alpha(s) \otimes 1 +
      s\otimes \alpha\,.
    \end{array}
  \right.
  \]
  Therefore, the natural mappings $S\to U$ and $L\to U$ induce an
  $R$-algebra homomorphism from $S\rtimes L$ to $U$. This homomorphism
  vanishes on $I$ whence the $R$-algebra homomorphism \eqref{eq:2}.

  Besides, the universal property of $U$ stated in \cite[Section 2,
  p. 110]{MR1696093} yields an $R$-algebra homomorphism
  \begin{equation}
    \label{eq:3}
    U \to (S\rtimes L)/I
  \end{equation}
  induced by the natural mappings $S\to (S\rtimes L/I)$ and $L\to
  (S\rtimes L)/I$. In view of the behaviour of \eqref{eq:2} and
  \eqref{eq:3} on the respective images of $S\cup L$, these algebra
  homomorphisms are inverse to each other.

  \medskip

  (2) It is proved in \cite[Lemma 4.1]{MR0154906} that $U$ is
  projective as a left $S$-module. Consider the increasing filtration
  of $U$ by the left $S$-submodules
  \[
  0 \subseteq F_0U \subseteq F_1U \subseteq \cdots
  \]
  where $F_pU$ is the image of $\oplus_{i=0}^p
  S\otimes L^{\otimes i}$ in $U$, for all $p\in \mathbb N$. In view
  of the equality
  \[
  \alpha s = s\alpha + \alpha(s)\,,
  \]
  in $U$ for all $s\in S$ and $\alpha\in L$, the left $S$-module
  $F_pU$ is also a right $S$-submodule of $U$, and $F_pU/F_{p-1}U$ is
  a symmetric $S$-bimodule for all $p\in \mathbb N$. Therefore, the
  considerations used in the proof of \cite[Lemma 4.1]{MR0154906} may
  be adapted in order to prove that $U$ is projective as a right
  $S$-module.
\end{proof}

\subsection{Basic constructions of $U$-modules}
\label{sec:basic-constr-u}

Left $S\rtimes L$-modules are identified with $R$-modules $N$ endowed
with a left $S$-module structure, and a left $L$-module structure such
that, for all $n\in N$, $\alpha\in L$ and $s\in S$,
\[
\alpha \cdot (s \cdot n) = \alpha(s)\cdot n + s\cdot (\alpha \cdot
n)\,.
\]
Left $U$-modules are identified with left $S\rtimes L$-modules $N$
such that, for all $n\in N$, $\alpha \in L$ and $s\in S$,
\[
s\cdot (\alpha\cdot n) = (s\alpha)\cdot n\,.
\]
Recall that the action of $L$ endows $S$ with a left $U$-module
structure.

Right $S\rtimes L$-modules are identified with the $R$-modules $M$
endowed with a right $S$-module structure and a right $L$-module
structure such that, for all $m\in M$, $\alpha\in L$ and $s\in S$,
\[
(m \cdot \alpha) \cdot s = m\cdot \alpha(s) + (m\cdot s) \cdot
\alpha\,.
\]
Right $U$-modules are identified with right $S\rtimes L$-modules $M$
such that, for all $m\in M$, $s\in S$ and $\alpha\in L$,
\[
(m\cdot s)\cdot \alpha = m \cdot (s\alpha)\,.
\]

The following constructions are classical. The corresponding
 $U$-module structures are introduced in \cite[Section 2]{MR1696093}.

Let $M,M'$ be right $S\rtimes L$-modules. Let $N,N'$ be left $S\rtimes
L$-module. Then,
\begin{itemize}
\item $N$ is a right $S\rtimes L$-module for the right $L$-module
  structure such that, for all $n\in N$, $s\in S$ and $\alpha \in L$,
  \begin{equation}
    \label{eq:45}
    n\cdot s = s\cdot n\ \text{and}\ n\cdot \alpha = -\alpha \cdot n\,,
  \end{equation}
\item $\Hom_S(N,N')$ is a left $S\rtimes L$-module for the left
  $L$-module structure such that, for all $f\in \Hom_S(N,N')$,
  $n\in N$ and $\alpha \in L$,
  \begin{equation}
    \label{eq:46}
    (\alpha \cdot f) (n) = \alpha \cdot f(n) - f(\alpha \cdot n)\,,
  \end{equation}
  moreover, this is a left $U$-module structure if $N$ and $N'$ are
  left $U$-modules,
\item $\Hom_S(M,M')$ is a left $S\rtimes L$-module for the left
  $L$-module structure such that, for all $f\in \Hom_S(M,M')$,
  $m\in M$ and $\alpha\in L$,
  \begin{equation}
    \label{eq:47}
    (\alpha \cdot f) (m) = - f(m) \cdot \alpha + f(m \cdot \alpha)\,,
  \end{equation}
\item $\Hom_S(N,S)$ is a right $S\rtimes L$-module for the right
  $L$-module structure such that, for all $f\in \Hom_S(N,S)$, $n\in N$
  and $\alpha\in L$,
  \begin{equation}
    \label{eq:48}
    (f\cdot \alpha) (n) = -\alpha(f(n)) +f(\alpha\cdot n)\,,
  \end{equation}
\item $N \otimes_S N'$ is a left $S\rtimes L$-module for the left
  $L$-module structure such that, for all $n\in N$, $n'\in N$ and
  $\alpha \in L$,
  \begin{equation}
    \label{eq:49}
    \alpha \cdot (n \otimes n') = \alpha \cdot n \otimes n' + n \otimes \alpha \cdot
    n'\,,
  \end{equation}
  moreover, this is a left $U$-module structure if $N$ and $N'$ are
  left $U$-modules,
\item $M\otimes_S N$ is a left $S\rtimes L$-module for the left
  $L$-module structure such that, for all $m\in M$, $n\in N$ and
  $\alpha\in L$,
  \begin{equation}
    \label{eq:50}
    \alpha \cdot (m\otimes n) = -m\cdot \alpha \otimes n + m \otimes
    \alpha \cdot n\,.
  \end{equation}
\end{itemize}

\subsection{The functor $G=\Hom_{S^e}(S,-)\colon \Mod(U^e)\to \Mod(U)$}
\label{sec:funct-gcol-modu-1}

Given $M\in \Mod(U^e)$, recall that
\[
M^S = \{m \in M\ |\ \text{for $s\in S$,}\ \ (s\otimes 1-1\otimes s) \cdot m
= 0\}\,.
\]
This is a symmetric $S^e$-submodule of $M$. Recall also the canonical
isomorphisms that are inverse to each other
\begin{equation}
  \label{eq:5}
  \begin{array}{rcl}
    M^S & \leftrightarrow & \mathrm{Hom}_{S^e}(S,M)\\
    m & \mapsto & (s\mapsto (s\otimes 1)\cdot m)\\
    \varphi(1) & \mapsfrom & \varphi\,.
  \end{array}
\end{equation}
\begin{lem}
  \label{sec:funct-gcol-modu}
  Let $M\in \Mod(U^e)$. Then,
  \begin{enumerate}
  \item $M^S$ is a left $U$-module such that, for all $m\in M^S$ and
    $\alpha\in L$,
    \begin{equation}
      \label{eq:6}
      \alpha \cdot m := (\alpha \otimes 1 - 1 \otimes \alpha) \cdot m\,;
    \end{equation}
  \item the corresponding left $U$-module structure on
    $\Hom_{S^e}(S,M)$ (under the identification (\ref{eq:5})) is such
    that, for all $\varphi\in \Hom_{S^e}(S,M)$, $\alpha\in L$ and
    $s\in S$,
    \[
    (\alpha \cdot \varphi) (s) =
    (\alpha \otimes 1 - 1 \otimes \alpha) \cdot \varphi(s)
    - \varphi(\alpha(s))\,.
    \]
  \end{enumerate}
\end{lem}
\begin{proof}
  (1) Given all $s\in S$ and $\alpha\in L$, denote
  \[
  s\otimes 1 - 1\otimes
  s\in U^e\ \text{and}\ \alpha \otimes 1 - 1 \otimes \alpha\in U^e
  \]
  by $ds$ and $d\alpha$, respectively; in particular
  \[
  d\alpha \cdot ds = ds\cdot d\alpha + d(\alpha(s))\,,
  \]
  and, for all $m\in M^S$,
  \[
  \begin{array}{rcl}
    ds\cdot (d\alpha \cdot m)
    & = & d\alpha \cdot (ds \cdot m) - d(\alpha(s)) \cdot m \\
    & = & 0\,,
  \end{array}
  \]
  which proves that $d\alpha\cdot m \in M^S$.  Therefore, \eqref{eq:6}
  defines a left $L$-module structure on $M^S$. Now, for all
  $m\in M^S$, $s\in S$ and $\alpha\in L$,
  \[
  \begin{array}{rcl}
    \alpha \!\cdot\! (s\!\otimes\! 1) \!\cdot\! m
    & \!\!\!\!= &
              \!\!\!\! d\alpha \!\cdot\! (s\!\otimes\! 1) \!\cdot\! m =
              (\alpha(s) \!\otimes\! 1 + s\alpha \!\otimes\! 1 - s\!\otimes\! \alpha)
              \!\cdot\! m \\
    & \!\!\!\!= &
              \!\!\!\! (\alpha(s) \!\otimes\! 1) \!\cdot\! m + (s \!\otimes\! 1) (\alpha
              \!\otimes\! 1 - 1 \!\otimes\! \alpha) \!\cdot\! m \\
    & \!\!\!\!= &
              \!\!\!\! (\alpha(s) \!\otimes\! 1)\!\cdot\! m + (s\!\otimes\! 1) \!\cdot\! (\alpha
              \!\cdot\! m) \\
    \\
    (s\!\otimes\! 1) \!\cdot\! (\alpha \!\cdot\! m)
    & \!\!\!\!= &
              \!\!\!\! (s\!\otimes\! 1) \!\cdot\! (\alpha \!\otimes\! 1 - 1 \!\otimes\! \alpha)
              \!\cdot\! m =
              (s \alpha \!\otimes\! 1) \!\cdot\! m - (s \!\otimes\! 1) \!\cdot\! (1
              \!\otimes\! \alpha) \!\cdot\! m \\
    & \!\!\!\!= &
              \!\!\!\! (s\alpha \!\otimes\! 1) \!\cdot\! m - (1\!\otimes\! \alpha) \!\cdot\!
              (s\!\otimes\! 1) \!\cdot\! m =
              (s\alpha \!\otimes\! 1) \!\cdot\! m - (1\!\otimes\! \alpha) \!\cdot\! (1
              \!\otimes\! s) \!\cdot\! m \\
    & \!\!\!\!= &
              \!\!\!\! (s\alpha \!\otimes\! 1 - 1 \!\otimes\! s\alpha) \!\cdot\! m =
              (s\alpha) \!\cdot\! m\,.
  \end{array}
  \]
  Hence, this left $L$-module structure on $M^S$ is a left $U$-module
  structure.

  \medskip
  
  (2) By definition, $\Hom_{S^e}(S,M)$ is endowed with the left
  $U$-module structure such that (\ref{eq:5}) is an isomorphism in
  $\Mod(U)$. Let $\varphi \in \Hom_{S^e}(S,M)$, $\alpha \in L$ and
  $s\in S$. Then,
  \[
  \begin{array}{rcl}
    (\alpha \cdot \varphi)(s)
    & = &
          (1 \otimes s) \cdot (\alpha \cdot \varphi(1)) =
          ((1 \otimes s) (\alpha \otimes 1 - 1 \otimes \alpha)) \cdot
          \varphi(1) \\
    & = &
          (\alpha \otimes s - 1 \otimes s\alpha - 1 \otimes \alpha(s))
          \cdot \varphi(1) \\
    & = &
          ((\alpha \otimes 1 - 1 \otimes \alpha)(1 \otimes s) - 1
          \otimes \alpha(s)) \cdot \varphi(1) \\
    & = &
          \alpha \cdot (1 \otimes s) \cdot \varphi(1) - (1 \otimes
          \alpha(s)) \cdot \varphi(1) =
          \alpha \cdot \varphi(s) - \varphi(\alpha(s))\,.
  \end{array}
  \]
\end{proof}
Thus, the assignment $M\mapsto M^S$ defines a functor
\begin{equation}
  \label{eq:7}
  \begin{array}{crcl}
    G \colon & \Mod(U^e) & \to & \Mod(U) \\
             & M & \mapsto & M^S\,.
  \end{array}
\end{equation}

\subsection{The functor $F = U \otimes_S-\colon \Mod(U) \to \Mod(U^e)$}
\label{sec:functor-fcolon-modu}

Let $N\in \Mod(U)$. In view of \cite[(2.4)]{MR1696093},
$U_U \otimes_S N$ is a right $U$-module such that, for all $u\in U$,
$n\in N$, $s\in S$ and $\alpha\in L$,
\[
(u \otimes n) \cdot s = u \otimes sn = us\otimes n\ \ \text{and}\ \ (u
\otimes n) \cdot \alpha = u\alpha \otimes n - u \otimes \alpha \cdot
n\,.
\]
Besides, $U \otimes_S N$ is a left $U$-module such that, for all $u,u'\in
U$ and $n\in N$,
\[
u' \cdot (u \otimes n) = u'u \otimes n\,.
\]
Therefore, $U\otimes_S N$ is a $U$-bimodule, and hence a left
$U^e$-module. These considerations define a functor
\begin{equation}
  \label{eq:8}
  \begin{array}{crcl}
    F \colon & \Mod(U) & \to & \Mod(U^e) \\
             & N & \mapsto & U \otimes_S N\,.
  \end{array}
\end{equation}

\subsection{The adjunction between $F$ and $G$}
\label{sec:adjunction-between-f}

\begin{prop}
  The functors $F = U\otimes_S-$ and $G=\Hom_{S^e}(S,-)$ introduced in
  \ref{sec:funct-gcol-modu-1} and \ref{sec:functor-fcolon-modu} form
  an adjoint pair
  \[
  \xymatrix{
    \Mod U \ar@<-3pt>[d]_F \\ \Mod U^e \ar@<-3pt>[u]_G\,.
    }
  \]
  In particular, there is a functorial isomorphism, for all
  $M\in \Mod(U^e)$ and $N\in \Mod(U)$,
  \[
  \Hom_U(N,G(M)) \xrightarrow{\sim} \Hom_{U^e}(F(N),M)\,.
  \]
\end{prop}
\begin{proof}
  Given $f\in \Hom_U(N,G(M))$, denote by $\Phi(f)$ the well-defined mapping
  \[
  \begin{array}{rcl}
    U \otimes_S N & \to & M \\
    u \otimes n & \mapsto & (u \otimes 1)\cdot f(n)\,;
  \end{array}
  \]
  consider $F(N)$ ($=U \otimes_S N$) as a $U$-bimodule. Then, for all
  $u,u'\in U$, $n\in N$, $s\in S$ and $\alpha\in L$, 
  \[
  \begin{array}{rcl}
    \Phi(f) (u'\cdot (u \otimes n))
    & = &
          \Phi(f) (u'u\otimes n) 
          =
          (u'u \otimes 1) \cdot f(n) \\
    & = &
          (u' \otimes 1) \cdot \Phi(f)(u \otimes n)\,, \\
    \\
    \Phi(f)((u\otimes n)\cdot s)
    & = &
          \Phi(f)(u \otimes s\cdot n) 
          =
          (u \otimes 1) \cdot f(s\cdot n) \\
    & = &
          (u \otimes 1) \cdot ((1 \otimes s)\cdot f(n))
          =
          ((1 \otimes s) \cdot (u \otimes 1)) \cdot f(n) \\
    & = &
          (1 \otimes s) \cdot \Phi(f) (u \otimes n) 
          =
          (\Phi(f)(u\otimes n))\cdot s \\
    \\
    \Phi(f) ((u \otimes n) \cdot \alpha)
    & = &
          \Phi(f) (u \alpha \otimes n - u \otimes \alpha \cdot n) \\
    & = &
          (u \alpha \otimes 1) \cdot f(n) - (u \otimes 1)\cdot
          f(\alpha \cdot n) \\
    & = &
          (u \alpha \otimes 1) \cdot f(n) - (u\otimes 1) \cdot (\alpha
          \otimes 1 - 1 \otimes \alpha) \cdot f(n) \\
    & = &
          (u \otimes \alpha) \cdot f(n)
          =
          (1 \otimes \alpha) \cdot \Phi(f) (u \otimes n) \\
    & = &
          (\Phi(f)(u \otimes n))\cdot \alpha\,;
  \end{array}
  \]
  in other words,
  \[
  \Phi(f) \in \Hom_{U^e}(F(N),M)\,.
  \]
  Given $g\in \Hom_{U^e}(F(N),M)$, then, for all $n\in N$ and $s\in
  S$,
  \[
    (s \otimes 1 - 1 \otimes s) \cdot g(1 \otimes n)
    =
          g(s \otimes_S n - 1 \otimes_S s\cdot n) \\
          =
          0\,;
  \]
  hence, denote by $\Psi(g)$ the well-defined mapping
  \[
  \begin{array}{rcl}
    N & \to & M^S \\
    n & \mapsto & g(1\otimes n)\,.
  \end{array}
  \]
  Therefore, for all $n\in N$, $s\in S$ and $\alpha \in L$,
  \[
  \begin{array}{rcl}
    \Psi(g)(s\cdot n)
    & = &
          g(1 \otimes s\cdot n)
      = 
          g(s\otimes n)
      = 
          g((s \otimes 1)\cdot(1 \otimes n)) \\
    & = &
          (s \otimes 1) \cdot g(1 \otimes n)
      = 
          (s\otimes 1) \cdot \Psi(g)(n) \\
    \\
    \Psi(g)(\alpha \cdot n)
    & = &
          g(1 \otimes \alpha \cdot n) 
          =
          g(\alpha \otimes n - (1\otimes \alpha) \cdot (1 \otimes n))
    \\
    & = &
          (\alpha \otimes 1) \cdot g(1 \otimes n) - (1 \otimes \alpha)
          \cdot g(1\otimes n) 
          =
          \alpha \cdot \Psi(g)(n)\,;
  \end{array}
  \]
  in other words,
  \[
  \Psi(g) \in \Hom_U(N,G(M))\,.
  \]
  By construction, $\Psi$ and $\Phi$ are inverse to each other.
\end{proof}

\subsection{Compatible left $S\rtimes L$-modules}
\label{sec:10}

As explained in Section~\ref{sec:main-results}, the main results of
this article are expressed in terms of the action of $L$ on
$\Ext^\bullet_{S^e}(S,S^e)$ by Lie derivatives and will be presented in
Section~\ref{sec:left-u-module}. Although this action does not define
a $U$-module structure on $\Ext^\bullet_{S^e}(S,S^e)$, it satisfies
some compatibility with the $S$-module structure. The actions of $L$
satisfying such a compatibility have specific properties that are used
in the rest of the article and which are summarised below.

Define a \emph{compatible} left $S\rtimes L$-module as a left
$S \rtimes L$-module $N$ such that, for all $n\in N$, $\alpha\in L$
and $s\in S$, the elements $s\alpha\in L$ and $\alpha(s)\in S$ satisfy
  \begin{equation}
    \label{eq:4}
    (s\alpha) \cdot n = s\cdot (\alpha \cdot n) -\alpha(s) \cdot n\,.
  \end{equation}
Note that a left $S\rtimes L$-module is both compatible and a left
$U$-module if and only if $L$ acts trivially, that is, by the zero action.

The two following lemmas present the properties of compatible left
$S\rtimes L$-modules used in the rest of the article.
\begin{lem}
  \label{sec:left-right-u}
  Let $M$ be a right $U$-module. Let $N$ be a compatible left
  $S\rtimes L$-module.  Then:
  \begin{enumerate}
  \item The right $S\rtimes L$-module $N^\vee =\Hom_S(N,S)$ is a right
    $U$-module,
  \item The left $S\rtimes L$-module $\Hom_S(N^\vee,M)$ is a left
    $U$-module,
  \item The left $S\rtimes L$-module $M\otimes_S N$ is a left
    $U$-module,
  \item The following canonical mapping is a morphism of left
    $U$-modules,
    \[
    \begin{array}{crcl}
      \theta \colon 
      & M \otimes_S N & \to & \Hom_S(N^\vee,M) \\
      & m \otimes n & \mapsto & \left( \theta_{m\otimes n} \colon
                                \varphi \mapsto m \cdot \varphi(n) \right)\,.
    \end{array}
    \]
  \end{enumerate}
\end{lem}
\begin{proof}
  (1) Given $\varphi \in N^\vee$, $s\in S$ and $\alpha\in L$, then
  \[
  \varphi \cdot (s\alpha) = (\varphi \cdot s) \cdot \alpha\,.
  \]
  Indeed, for all $n\in N$,
  \[
  \begin{array}{rcl}
    (\varphi \cdot (s\alpha)) (n)
    & = & -(s\alpha)(\varphi(n)) + \varphi((s\alpha)\cdot n) \\
    & = & -s \, \alpha(\varphi(n)) + \varphi ( s\cdot (\alpha\cdot n) -
          \alpha(s) \cdot n) \\
    & = & -s \, \alpha(\varphi(n)) + s \, \varphi(\alpha\cdot n) - \alpha(s)
          \varphi(n) \\
    & = & ((\varphi \cdot \alpha)\cdot s)(n) - (\varphi \cdot
          \alpha(s))(n) \\
    & = & ((\varphi\cdot s)\cdot \alpha)(n)\,.          
  \end{array}
  \]
  
  \medskip

  (2) This
  is precisely \cite[(2.3)]{MR1696093}.

  \medskip

  (3) The $S\rtimes L$-module structure of $M\otimes_S N$ is described
  in \eqref{eq:50}. Given $m\in M$, $n\in N$, $s\in S$ and
  $\alpha\in L$, then
  \[
  \begin{array}{rcl}
    (s\alpha)\cdot (m\otimes n)
    & \!\!\!\!=& \!\!-m\cdot (s\alpha) \otimes n + m \otimes (s\alpha)\cdot n \\
    & \!\!\!\!= & \!\!-(m \cdot \alpha) \cdot s \otimes n + m \cdot \alpha(s)
          \otimes n
          
          + m \otimes s\cdot (\alpha\cdot n) - m \otimes
          \alpha(s) \cdot n \\
    & \!\!\!\!= & \!\!s\cdot (\alpha \cdot (m\otimes n))\,.
  \end{array}
  \]

  \medskip

  (4)  It suffices to prove that the given mapping is $L$-linear. Let
  $m\in M$, $n\in N$, $\alpha\in L$ and $\varphi\in
  \Hom_S(N,S)$. Then,
  \[
  \begin{array}{rcl}
    (\alpha\cdot \theta_{m\otimes n})(\varphi)
    & \!\!\!\!= & \!\!-\theta_{m\otimes n}(\varphi)\cdot \alpha + \theta_{m\otimes
          n}(\varphi \cdot \alpha)
    = -(m\cdot \varphi(n))\cdot \alpha + m \cdot (\varphi\cdot
          \alpha)(n) \\
    & \!\!\!\!= & \!\!-((m\cdot \alpha)\cdot \varphi(n) - m \cdot
          \alpha(\varphi(n))) 
          + m \cdot (-\alpha(\varphi(n)) + \varphi(\alpha \cdot n)) \\
    & \!\!\!\!= & \!\!-(m\cdot \alpha) \cdot \varphi(n) + m \cdot
          \varphi(\alpha\cdot n)
    = \theta_{\alpha\cdot (m\otimes n)}(\varphi)\,;
  \end{array}
  \]
  thus, $\alpha\cdot \theta_{m\otimes n} = \theta_{\alpha\cdot
    (m\otimes n)}$.
\end{proof}

Any left $S\rtimes L$-module $N$ may be considered as a symmetric
$S$-bimodule, or equivalently a right $S^e$-module, such that, for all
$n\in N$ and $s,s'\in S$
\[
n\cdot (s\otimes s') = (ss')\cdot n\,;
\]
accordingly, $N\otimes_{S^e}U^e$ is a right $U^e$-module in a natural way.
\begin{lem}
  \label{sec:1}
  Let $N$ be a compatible left $S\rtimes L$-module
  \begin{enumerate}
  \item The right $U^e$-module $N\otimes_{S^e}U^e$ is actually a
    $U$-$U^e$-bimodule such that for all $n\in N$, $u,v\in U$ and $\alpha\in L$,
      \[
      \alpha \cdot (n \otimes (u \otimes v))
      =
      \alpha \cdot n \otimes (u \otimes v)
      + n \otimes ((\alpha\otimes 1 -1 \otimes \alpha) \cdot (u
      \otimes v))\,.
      \]
  \item Let $M$ be a right $U$-module. Then, there exists an
    isomorphism of left $U^e$-modules
    \[
    \begin{array}{rcl}
      F(M\otimes_S N) & \to & M \otimes_U (N \otimes_{S^e}U^e) \\
      v \otimes (m \otimes n) & \mapsto & m \otimes (n \otimes (1
                                          \otimes v))\,.
    \end{array}
    \]
  \end{enumerate}
\end{lem}
\begin{proof}
  (1) Following part (3) of Lemma~\ref{sec:left-right-u}, there is a
  left $U$-module structure on $U \otimes_S N$ such that, for all
  $\alpha \in L$, $v\in U$ and $n\in N$,
  \[
  \alpha \cdot (v \otimes n) = -v\alpha \otimes n + v \otimes \alpha
  \cdot n\,.
  \]
  Therefore, there is a left $U$-module structure on
  $(U\otimes_S N) \otimes_S U$ (see \eqref{eq:49}) such that, for all
  $\alpha\in L$, $n\in N$ and $u,v\in U$,
  \[
  \begin{array}{rcl}
  \alpha \cdot ((v \otimes n) \otimes u)
    & = &
          \alpha \cdot (v \otimes n) \otimes u + (v \otimes n) \otimes \alpha
           u \\
    & = &
          -(v\alpha \otimes n )\otimes u + (v \otimes \alpha \cdot n)
          \otimes u + (v\otimes n) \otimes \alpha u\,.
  \end{array}
  \]
  Under the canonical identification
  \[
  \begin{array}{rcl}
    N \otimes_{S^e} U^e & \to & (U \otimes_S N) \otimes_S U \\
    n \otimes (u \otimes v) & \mapsto & (v \otimes n) \otimes u\,,
  \end{array}
  \]
  $N \otimes_{S^e} U^e$ inherits of a left $U$-module structure which
  is the one claimed in the statement.

  Now, $N\otimes_{S^e} U^e$ inherits a right $U^e$-module structure
  from $U^e$. This structure is compatible with the left $U$-module
  discussed previously so as to yield a left
  $U\otimes (U^e)^\mathrm{op}$-module structure.

  \medskip

  (2) Due to (1), there is a right $U^e$-module structure on
  $M\otimes_U (N\otimes_{S^e} U^e)$. It is considered here as a left
  $U^e$-module structure such that, for all $u,v,u',v'\in U$, $m\in M$
  and $n\in N$,
  \begin{equation}
    \label{eq:35}
    (u' \otimes v') \cdot (m \otimes (n \otimes (u \otimes v)))
    =
    m \otimes (n \otimes (uv' \otimes u'v))\,;
  \end{equation}
  For the ease of reading, note that in $F(M\otimes_S N)$,
  \begin{equation}
    \label{eq:9}
    \begin{array}{rcl}
      (u \otimes 1) \cdot (v \otimes m \otimes n)
      & \!\!\!\!= &
                    \!\!\!\!uv \otimes m \otimes n \\
      (1 \otimes \alpha) \cdot (v \otimes m \otimes n)
      & \!\!\!\!= &
                    \!\!\!\! v\alpha \otimes m \otimes n 
                    + v \otimes m\cdot \alpha \otimes n 
                    - v \otimes m \otimes \alpha \cdot n\,,
    \end{array}
  \end{equation}
  and, in $M\otimes_U (N \otimes_{S^e} U^e)$,
  \begin{equation}
    \label{eq:10}
    m \cdot \alpha \otimes n \otimes u \otimes v
    =
    m \otimes \alpha \cdot n \otimes u \otimes v \\
    + m \otimes n \otimes \alpha u \otimes v \\
    - m \otimes n \otimes u \otimes v \alpha\,.
  \end{equation}
  The $R$-linear mapping from $U\otimes M \otimes N$ to $M \otimes_U
  (N \otimes_{S^e}U^e)$ given by
  \[
  v \otimes m \otimes n \mapsto m \otimes (n \otimes (1 \otimes v))
  \]
  induces a morphism of $S$-modules from $U \otimes_S (M\otimes_S N)$ to
  $M \otimes_U (N \otimes_{S^e} U^e)$ such as in the statement of the
  lemma. Denote it by $\Psi'$:
  \[
  \Psi' \colon U \otimes_S (M \otimes_S N) \to M \otimes_U (N
  \otimes_{S^e} U^e)\,.
  \]
  This is a morphism of left $U^e$-modules. Indeed, for all $u,v\in U$,
  $m\in M$, $n\in N$ and $\alpha\in L$,
  \[
  \begin{array}{rcl}
    \Psi'((u \!\otimes\! 1) \cdot (v \!\otimes\! m \!\otimes\! n))
    & \!\!\!\!\!\!\!\!= & \!\!\!\!\!\!\!\!
                          \Psi'(uv \!\otimes\! m \!\otimes\! n) =
                          m \!\otimes\! n \!\otimes\! 1 \!\otimes\! uv \\
    & \!\!\!\!\!\!\!\!\underset{~(\ref{eq:35})}
      = &\!\!\!\!\!\!\!\!
          (u \!\otimes\! 1) \cdot \Psi'(v \!\otimes\! m \!\otimes\! n) \\
    \\
    \Psi'((1 \!\otimes\! \alpha) \cdot (v \!\otimes\! m \!\otimes\! n))
    & \!\!\!\!\!\!\!\!= & \!\!\!\!\!\!\!\!
                          \Psi'(v \alpha \!\otimes\! m \!\otimes\! n + v \!\otimes\! m \cdot \alpha
                          \!\otimes\! n 
                          - v \!\otimes\! m \!\otimes\! \alpha \cdot n) \\
    & \!\!\!\!\!\!\!\!= & \!\!\!\!\!\!\!\!
                          m \!\otimes\! n \!\otimes\! 1 \!\otimes\! v \alpha
                          + m \cdot \alpha \!\otimes\! n \!\otimes\! 1 \!\otimes\! v
                          - m \!\otimes\! \alpha \cdot n \!\otimes\! 1 \!\otimes\! v \\
    & \!\!\!\!\!\!\!\!\underset{\eqref{eq:10}}{=} &\!\!\!\!\!\!\!\!
                                                    m \!\otimes\! n \!\otimes\! \alpha \!\otimes\! v
    \\
    & \!\!\!\!\!\!\!\!\underset{~(\ref{eq:35})}
      = &\!\!\!\!\!\!\!\!
          (1 \!\otimes\! \alpha) \cdot \Psi'(v \!\otimes\! m \!\otimes\! n)\,.
  \end{array}
  \]
  Consider the following morphism of $S$-modules
  \[
  \begin{array}{crcl}
    \phi \colon & M \otimes_S (N \otimes_{S^e} U^e) & \to & F(M
                                                            \otimes_S N)
    \\
                & m \otimes (n \otimes (u \otimes v)) & \mapsto & (1 \otimes u)
                                                                  \cdot (v \otimes m
                                                                  \otimes n)\,.
  \end{array}
  \]
  Given $m\in M$, $n\in N$, $u,v\in U$ and $\alpha \in L$, then the
  image under $\phi$ of the term
  \[
  m \otimes \alpha \cdot n \otimes u \otimes v
  +  m \otimes n \otimes \alpha u \otimes v
  -  m \otimes n \otimes u \otimes v \alpha
  \]
  is equal to
  \[
  (1 \otimes u) \cdot (v \otimes m \otimes \alpha \cdot n)
  + (1\otimes \alpha u) \cdot (v \otimes m \otimes n)
  - (1 \otimes u) \cdot (v \alpha \otimes m \otimes n)\,,
  \]
  which is equal to
  \[
  (1 \otimes u) \cdot (v \otimes m \otimes \alpha \cdot n)
  + (1 \otimes u) \cdot (1\otimes \alpha) \cdot (v \otimes m \otimes n)
  - (1 \otimes u) \cdot (v\alpha \otimes m \otimes n)\,.
  \]
  In view of \eqref{eq:9}, this is equal to
  \[
  (1 \otimes u) \cdot (v \otimes m\cdot \alpha \otimes n) = \phi( m
  \cdot \alpha \otimes (n \otimes (u \otimes v)))\,.
  \]
  Thus, $\phi$ induces a morphism of $S$-modules
  \[
  \begin{array}{crcl}
    \Phi' \colon & M \otimes_U (N \otimes_{S^e} U^e) & \to & F(M
                                                            \otimes_S N)
    \\
                & m \otimes (n \otimes (u \otimes v)) & \mapsto & (1 \otimes u) \cdot
                                                                  (v \otimes m \otimes
                                                                  n)\,.
  \end{array}
  \]
  It appears that $\Phi'$ is left and right inverse for $\Psi'$. Indeed,
  \begin{itemize}
  \item $\Phi'\circ \Psi' = \mathrm{Id}_{F(M \otimes_S N)}$, and
  \item for all $u,v\in U$, $m\in M$ and $n\in N$,
    \[
    \begin{array}{rcll}
      \Psi' \circ \Phi'(m \otimes n \otimes u \otimes v)
      & = &
            \Psi'((1\otimes u) \cdot (v \otimes m \otimes n)) \\
      & = &
            (1\otimes u) \cdot \Psi'(v \otimes m \otimes n)
      & \text{\tiny ($\Psi'$ is $U^e$-linear)} \\
      & = &
            (1\otimes u) \cdot (m \otimes n \otimes 1 \otimes v) \\
      & = &
            m \otimes n \otimes u \otimes v\,.
    \end{array}
    \]
  \end{itemize}
\end{proof}

\subsection{Invertible $U$-bimodules}
\label{sec:invert-u-bimod-1}
The following result is used in Section~\ref{sec:proof-main-theorem}
in order to prove that $\Ext^i_{U^e}(U,U^e)$ is invertible as a
$U$-bimodule, under suitable conditions.
\begin{prop}
  \label{sec:4}
  Let $R$ be a commutative ring. Let $(S,L)$ be a Lie-Rinehart algebra
  over $R$. Denote by $U$ its enveloping algebra.
  Let $N$ be a left $U$-module. Assume that $N$ is invertible as an
  $S$-module. Then $F(N)$ is invertible as a $U$-bimodule.
\end{prop}

The subsection is devoted to the proof of this proposition.
Given a left $U$-module $N$, then $F(N) = U \otimes_S N$ as left
$U$-modules. Hence, there is a functorial isomorphism
\begin{equation}
  \label{eq:40}
  \theta \colon \Hom_S(N,U) \to \Hom_U(F(N),U)\,.
\end{equation}

Note:
\begin{itemize}
\item $\Hom_S(N,U)$ is a left $U$-module (see (\ref{eq:46})), and it
  inherits a right $U$-module structure from $U_U$; by construction,
  these two structures form a $U$-bimodule structure.
\item  $\Hom_U(F(N),U)$ is a $U$-bimodule because so are $F(N)$ and
  $U$. 
\item  $N \otimes_S \Hom_S(N,U)$ is a left $U$-module (see
  (\ref{eq:49})), and it inherits  a right $U$-module structure
  from $U_U$; by construction, these two structures form a
  $U$-bimodule structure.
\end{itemize}
\begin{lem}
  \label{sec:2}
  Let $N$ be a left $U$-module. Then,
  \begin{enumerate}
  \item $\theta \colon \Hom_S(N,U) \to \Hom_U(F(N),U)$ is an
    isomorphism in $\Mod(U^e)$,
  \item the mapping
    \[
    \begin{array}{cccc}
      \Phi \colon & N \otimes_S \Hom_S(N,U) & \to & F(N) \otimes_U
                                                    \Hom_U(F(N),U) \\
                  & n \otimes f & \mapsto & (1 \otimes n) \otimes \theta(f)
    \end{array}
    \]
    is an isomorphism in $\Mod(U^e)$, and
  \item the diagram
    \[
    \xymatrix{
      {\scriptstyle N \otimes_S \Hom_S(N,U)} \ar[r] \ar[d]_{\Phi} & U
      \ar@{=}[d] \\
      {\scriptstyle F(N) \otimes_U \Hom(F(N),U)} \ar[r] & U
    }
    \]
    with horizontal arrows given by evaluation, is commutative.
  \end{enumerate}
\end{lem}
\begin{proof}
  (1) By definition, $\theta$ is a morphism of right $U$-modules. It
  is also a morphism of left $U$-modules because, for all $n \in N$,
  $f \in \Hom_S(N,U)$, $u \in U$ and $\alpha \in L$,
  \[
  \begin{array}{rcl}
    \theta(\alpha \cdot f) (u \otimes n)
    & = &
          u (\alpha \cdot f)(n) \\
    & = &
          u (\alpha f(n) - f(\alpha \cdot n))
     = 
          \theta(f) (u \alpha \otimes n - u \otimes \alpha \cdot n) \\
    & = &
          \theta(f) ((u \otimes n) \cdot \alpha)
     = 
          (\alpha \cdot \theta(f))(u \otimes n)\,.
  \end{array}
  \]

  \medskip

  (2) By definition, $\Phi$ is a morphism of right $U$-modules. It is
  also a morphism of left $U$-modules because, for all $n \in N$,
  $f\in \Hom_S(N,U)$ and $\alpha \in L$,
  \[
  \begin{array}{rcl}
    \Phi( \alpha \cdot (n \otimes f))
    & = &
          \Phi( \alpha \cdot n \otimes f + n \otimes \alpha \cdot f)
    \\
    & = &
          (1 \otimes \alpha \cdot n) \otimes \theta(f)
          + (1 \otimes n) \otimes
          \underset{=\alpha \cdot \theta(f)}{\underbrace{\theta(\alpha
          \cdot f)}} \\
    & = &
          (1 \otimes \alpha \cdot n) \otimes \theta(f)
          + \underset{= \alpha \otimes n - 1 \otimes \alpha \cdot n}
          {\underbrace{(1 \otimes n) \cdot \alpha}} \otimes \theta(f)
    \\
    & = &
          (\alpha \otimes n) \otimes \theta(f) \\
    & = &
          \alpha \cdot \Phi( n \otimes f)\,.
  \end{array}
  \]
  In order to prove that $\Phi$ is bijective, consider the linear
  mapping
  \[
  \begin{array}{cccc}
    \psi \colon & F(N) \otimes_S \Hom_U(F(N),U) & \to & N \otimes_S
                                                        \Hom_S(N,U) \\
                & (u \otimes n) \otimes g & \mapsto & u \cdot (n \otimes
                                                      \theta^{-1}(g))\,.
  \end{array}
  \]
  Note that, for all $u \in U$, $\alpha \in L$, $n \in N$ and $g\in
  \Hom_U(F(N),U)$,
  \[
  \begin{array}{rcl}
    \psi((u \otimes n) \cdot \alpha \otimes g)
    & \!\!= &\!\!
              \psi((u \alpha \otimes n) \otimes g - (u \otimes \alpha
              \cdot n) \otimes g) \\
    & \!\!= &\!\!
              u \alpha \cdot (n \otimes \theta^{-1}(g)) - u \cdot (\alpha
              \cdot n \otimes \theta^{-1}(g)) \\
    & \!\!= &\!\!
              u \cdot (\alpha \cdot n \otimes \theta^{-1}(g) + n \otimes
              \alpha \cdot \theta^{-1}(g)) 
              - u \cdot (\alpha \cdot n \otimes \theta^{-1}(g)) \\
    & \!\!= &\!\!
              u \cdot (n \otimes \theta^{-1}(\alpha \cdot g))\ \ \ \ \
              \ \text{(see
              part (1))}\\
    & \!\!= &\!\!
              \psi((u \otimes n) \otimes \alpha \cdot g)\,.
  \end{array}
  \]
  Hence, $\psi$ induces a linear mapping
  \[
  \begin{array}{rrcl}
    \Psi \colon & F(N) \otimes_U \Hom_U(F(N),U) & \to & N \otimes_S
                                                        \Hom_S(N,U) \\
                & (u \otimes n) \otimes g & \mapsto & u \cdot (n \otimes \theta^{-1}(g))\,.
  \end{array}
  \]
  Now, by definition of $\Phi$ and $\Psi$,
  \[
  \Psi \circ \Phi = \mathrm{Id}_{N \otimes_S \Hom_S(N,U)}\,.
  \]
  Since
  \begin{itemize}
  \item $\Psi$ is a morphism of left $U$-modules by construction;
  \item as a left $U$-module, $F(N) \otimes_U \Hom_U(F(N),U)$ is
    generated by the image of 
    $(1 \otimes N) \otimes \Hom_U(F(N),U)$; and
  \item for all $n \in N$ and $g\in \Hom_U(F(N),U)$,
    \[
    \Phi \circ \Psi((1 \otimes n) \otimes g) = (1 \otimes n) \otimes
    g\,;
    \]
  \end{itemize}
  the following holds
  \[
  \Phi \circ \Psi = \mathrm{Id}_{F(N) \otimes_U \Hom_U(F(N),U)}\,.
  \]
  Altogether, these considerations show that $\Phi$ is an isomorphism
  in $\Mod(U^e)$.

  \medskip

  (3) The diagram is commutative by definition of $\Phi$.
\end{proof}

Like in the previous lemma, for all $N \in \Mod(U)$, $\Hom_S(N,U)$ is
a $U$-bimodule, and hence $\Hom_S(N,U) \otimes_S N$ is a $U$-bimodule
by means of (\ref{eq:49}) and the right $U$-module structure of $U$.
\begin{lem}
  \label{sec:3}
  Let $N$ be a left $U$-module. Then,
  \begin{enumerate}
  \item the mapping
    \[
    \begin{array}{rrcl}
      \Phi' \colon & \Hom_S(N,U) \otimes_S N & \to & \Hom_U(F(N),U)
                                                     \otimes_U F(N) \\
                   & f \otimes n & \mapsto & \theta(f) \otimes (1
                                             \otimes n)
    \end{array}
    \]
    is an isomorphism in $\Mod(U^e)$; and
  \item the diagram
    \[
    \xymatrix{
      {\scriptstyle \Hom_S(N,U) \otimes_S N} \ar[r] \ar[d]_{\Phi'} & U \ar@{=}[d] \\
      {\scriptstyle \Hom_U(F(N),U) \otimes_U F(N)} \ar[r] & U
    }
    \]
    with horizontal arrows  given by evaluation, is commutative.
  \end{enumerate}
\end{lem}
\begin{proof}
  (1) First, since $F(N) = U \otimes_S N$ in $\Mod(U^e)$, then
  \[
  \Hom_U(F(N),U) \otimes_U F(N) \cong \Hom_U(F(N),U) \otimes_S N
  \]
  as left $U$-modules. Under this identification, $\Phi'$ expresses as
  \[
  \Phi' \colon f \otimes n \mapsto \theta(f) \otimes n\,.
  \]
  Therefore, $\Phi'$ is bijective because so is $\theta$.

  Next, $\Phi'$ is a morphism of left $U$-modules because so is
  $\theta$. And it is a morphism of right $U$-modules because it is a
  morphism of right $S$-modules, and because, for all $f \in
  \Hom_S(N,U)$, $n \in N$ and $\alpha \in L$,
  \[
  \begin{array}{rcl}
    \Phi'((f \otimes n) \cdot \alpha)
    & = &
          \Phi'(f \cdot \alpha \otimes n - f \otimes \alpha \cdot n)
    \\
    & = &
          \underset{= \theta(f) \cdot \alpha}{\underbrace{\theta(f
          \cdot \alpha)}} \otimes (1 \otimes n)
          - \theta(f) \otimes (1 \otimes \alpha \cdot n) \\
    & = &
          \theta(f) \otimes \underset{= \alpha \otimes
          n}{\underbrace{\alpha \cdot (1 \otimes n)}} - \theta(f)
          \otimes (1 \otimes \alpha \cdot n) \\
    & = &
          \theta(f) \otimes ((1 \otimes n) \cdot \alpha) \\
    & = &
          (\theta(f) \otimes (1 \otimes n)) \cdot \alpha \\
    & = &
          \Phi'(f \otimes n) \cdot \alpha\,.    
  \end{array}
  \]
  This proves (1).

  \medskip

  (2) The diagram commutes by definition of $\Phi'$.
\end{proof}

It is now possible to prove the result announced at the beginning of
the subsection.
\begin{proof}[Proof of Proposition~\ref{sec:4}]
  Since $N$ is invertible as an $S$-module, then the following
  evaluation mappings are bijective
  \[
  \begin{array}{rcl}
    N \otimes_S \Hom_S(N,U)  \to  U 
    \ \ \text{and}\ \ 
    \Hom_S(N,U) \otimes_S N  \to  U\,.
  \end{array}
  \]
  According to Lemmas~\ref{sec:2} and \ref{sec:3}, the following
  evaluation mappings are isomorphisms of $U$-bimodules
  \[
  \begin{array}{rcl}
    F(N) \otimes_U \Hom_U(F(N),U)  \to  U 
    \ \ \text{and}\ \ 
    \Hom_U(F(N),U) \otimes_U F(N)  \to  U\,.
  \end{array}
  \]
  Thus, $F(N)$ is invertible as a $U$-bimodule.
\end{proof}

\section{The action of \texorpdfstring{$L$}{L} on the inverse dualising bimodule of \texorpdfstring{$S$}{S}}
\label{sec:left-u-module}

This section introduces an action of $L$ on
$\Ext^\bullet_{S^e}(S,S^e)$ by means of Lie derivatives, which is used
to describe $\Ext^\bullet_{U^e}(U,U^e)$ in the next section. When $S$ is
projective in $\Mod(R)$, then $\Ext^\bullet_{S^e}(S,-)$ is the
Hochschild cohomology $H^\bullet(S;-)$; in this setting, the Lie
derivatives on $H^\bullet(S;S)$ and $H_\bullet(S;S)$ are
defined in \cite[Section 9]{MR0154906} and have a well-known
expression in terms of the Hochschild resolution of $S$. For the needs
of the article, the definition is translated to arbitrary coefficients in
terms of any projective resolution of $S$ in $\Mod(S^e)$.

Hence, \ref{sec:data-proj-resol} introduces preliminary material,
\ref{sec:deriv-homot} deals with derivations on projective resolutions
of $S$ in $\Mod(S^e)$, \ref{sec:lie-derivatives} defines the Lie
derivatives, \ref{sec:acti-l-extb} presents the action of $L$ on
$\Ext^\bullet_{S^e}(S,S^e)$, and \ref{sec:part-case-calabi} discusses particular situations.

For the section, a projective resolution of $S$ in $\Mod(S^e)$ is
considered;
\[
(P^\bullet,d) \to S\,.
\]
Denote $S$ by $P^1$ and the augmentation mapping
$P^0\to S$ by $d^0$.  For all $M\in \Mod(S^e)$ and
$s\in S$, denote by $\lambda_s$ and $\rho_s$ the multiplication
mappings
\[
\begin{array}{lcc}
  & \lambda_s \colon M  \longrightarrow  M, & m  \longmapsto 
                                              (s\otimes
                                              1)
                                              \cdot
                                              m \\
  \text{and} \\
  & \rho_s \colon  M  \longrightarrow  M, & m  \longmapsto  (1\otimes
                                            s) \cdot m \,.
\end{array}
\]

\subsection{Data on the projective resolution}
\label{sec:data-proj-resol}

For all $s\in S$, the mappings
$\lambda_s,\rho_s$ on $P^\bullet$  are morphisms of
complexes of left $S^e$-modules and induce the same mapping
\[
\begin{array}{rcl}
  S & \to & S \\
  t & \mapsto & st
\end{array}
\]
in cohomology; Hence, there exists a morphism of graded left
$S^e$-modules
\begin{equation}
  \label{eq:11}
  k_s \colon P^\bullet \to P^\bullet[-1]
\end{equation}
such that
\begin{equation}
  \label{eq:12}
  \lambda_s - \rho_s = k_s \circ d + d \circ k_s\,.
\end{equation}

\begin{lem}
  \label{sec:data-proj-resol-1}
  Let $\partial \colon S\to S$ be an $R$-linear derivation. Let
  $\psi\colon P^\bullet \to P^\bullet$ be a morphism of complexes of
  $R$-modules such that
  \begin{itemize}
  \item $H^0(\psi) \colon S\to S$ is the zero mapping;
  \item there exists a morphism of graded left $S^e$-modules
    \[
    k \colon P^\bullet \to P^\bullet[-1]
    \]
    such that, for all $p\in P^\bullet$ and $s,t\in S$,
    \begin{equation}
      \label{eq:14}
      \psi((s \otimes t) \cdot p)
      =
      (s \otimes t) \cdot \psi(p)
      - (1 \otimes \partial)(s \otimes t) \cdot
      (k \circ d + d \circ k)(p)\,.
    \end{equation}
    Then, there exists a morphism of graded $R$-modules
    \[
    h \colon P^\bullet \to P^\bullet[-1]
    \]
    such that
    \begin{itemize}
    \item $\psi = h \circ d + d \circ h$; and
    \item for all $s,t\in S$ and $p\in P^\bullet$,
      \[
      h((s \otimes t) \cdot p)
      =
      (s \otimes t) \cdot h(p)
      -(1 \otimes \partial)(s \otimes t) \cdot k(p)\,.
      \]
    \end{itemize}
  \end{itemize}
\end{lem}
\begin{proof}
  The proof is an induction on $n\leqslant 1$, taking $h^1\colon
  S\to P^0$ equal to $0$. Let $n\leqslant 0$ and assume that there
  exist linear mappings, for all $j$ such that $n+1\leqslant j
  \leqslant 1$,
  \[
  h^j \colon P^j \to P^{j-1}\
  \]
  such that, for all $j$ satisfying $n+1\leqslant j \leqslant 0$,
  $p\in P^j$ and $s,t\in S$,
  \begin{equation}
    \label{eq:15}
    \begin{array}{rcl}
      \psi^j
      & = &
            h^{j+1} \circ d^j + d^{j-1} \circ h^j \\
      h^j((s \otimes t) \cdot p)
      & = &
            (s \otimes t) \cdot h^j(p) - (1 \otimes \partial)(s
            \otimes t) \cdot k^j(p)\,.
    \end{array}
  \end{equation}
  \[
  \xymatrix{
    P^n \ar[rr]^{d^n} \ar[d]_{\psi^n} && P^{n+1} \ar[d]_{\psi^{n+1}}
    \ar[lld]_{h^{n+1}} \ar[rr]^{d^{n+1}} && P^{n+2}
    \ar[rr]^{d^{n+2}} \ar[lld]_{h^{n+2}} && \cdots \\
    P^n \ar[rr]^{d^n} && P^{n+1} \ar[rr]^{d^{n+1}} && P^{n+2} \ar[rr]^{d^{n+2}} && \cdots \\
  }
  \]
  Let
  \[
 ((p_i,\varphi^i))_{i\in I}
  \]
  be a coordinate system of the projective left $S^e$-module
  $P^n$. That is, let $p_i\in P^n$ and
  $\varphi^i\in \Hom_{S^e}(P^n,S^e)$ for all $i\in I$, such that, for
  all $p\in P^n$,
  \[
  p = \sum_{i\in I} \varphi^i(p) \cdot p_i\,,
  \]
  where $\{i \in I \ |\ \varphi^i(p) \neq 0\}$ is finite.  Since
  $\psi \colon P^\bullet \to P^\bullet$ is a morphism of complexes, it
  follows from \eqref{eq:15} that, for all $i\in I$, there exists
  $p_i'\in P^{n-1}$ such that
  \begin{equation}
    \label{eq:16}
    \psi^n(p_i) =
    d^{n-1}(p_i') + h^{n+1} \circ d^n(p_i)\,.
  \end{equation}
  Denote by $h^n$ the linear mapping from $P^n$ to $P^{n-1}$ such
  that, for all $p \in P^n$,
  \[
  h^n(p)
  =
  \sum_{i\in I} \varphi^i(p)\cdot p_i' - (1
  \otimes \partial)(\varphi^i(p)) \cdot k^n(p_i)\,.
  \]
  Then, for all $p\in P^n$ and $s,t\in S$,
  \[
  \begin{array}{l}
    h^n((s\otimes t) \!\cdot\! p) \\ \\
    \!\!=\!\!
    \sum\limits_{i\in I}
    (s \otimes t) \!\cdot\! \varphi^i(p)\!\cdot\! p_i'
    \!-\! (s \otimes t) \!\cdot\! (1\otimes \partial)(\varphi^i(p))
    \!\cdot\! k^n(p_i)
    \!-\! (1 \otimes \partial)(s \otimes t) \!\cdot\! \varphi^i(p) \!\cdot\!
    k^n(p_i) \\ \\
    \!\!=\!\!
    (s \otimes t) \!\cdot\! h^n(p) \!-\! (1 \otimes \partial)(s \otimes
    t) \!\cdot\! k^n(\sum_{i\in I} \varphi^i(p) \!\cdot\! p_i) \\ \\
    \!\!=\!\! (s \otimes t) \!\cdot\! h^n(p)
    \!-\! (1 \otimes \partial) (s \otimes t) \!\cdot\! k^n(p)\,.    
  \end{array}
  \]
  Moreover,
  \[
  \psi^n = h^{n+1} \circ d^n + d^{n-1} \circ h^n\,.
  \]
  Indeed, for all $p\in P^n$,
  $p= \sum_{i\in I} \varphi^i(p) \cdot p_i$, and hence
  \[
  \begin{array}{l}
    d^{n-1} \circ h^n(p) + h^{n+1} \circ d^n(p) \\ \\
    \underset{\ \ \ \ }{\!\!=\!\!}
    \sum\limits_{i\in I}
              \varphi^i(p) \!\cdot\! d^{n-1}(p_i')
              - (1 \otimes \partial)(\varphi^i(p)) \!\cdot\! d^{n-1}\circ
              k^n(p_i)
          + h^{n+1}(\sum_{i\in I} \varphi^i(p) \!\cdot\! d^n(p_i)) \\ \\
     \underset{\eqref{eq:15}}{\!\!=\!\!}
          \sum\limits_{i\in I}
              \varphi^i(p) \!\cdot\! d^{n-1}(p_i')
              - (1\otimes \partial)(\varphi^i(p)) \!\cdot\! d^{n-1} \circ
              k^n(p_i) \\
    \null\hfill+ \varphi^i(p) \!\cdot\! h^{n+1} \circ d^n(p_i)
              - (1 \otimes \partial)(\varphi^i(p)) \!\cdot\! k^{n+1} \circ
              d^n(p_i) \\ \\
    \underset{~\eqref{eq:14}}{\!\!=\!\!}
            \sum\limits_{i\in I}
              \varphi^i(p) \!\cdot\! d^{n-1}(p_i')
              + \varphi^i(p) \!\cdot\! h^{n+1} \circ d^n(p_i)
              + \psi^n(\varphi^i(p) \!\cdot\! p_i)
              - \varphi^i(p) \!\cdot\! \psi^n(p_i) \\ \\
  \underset{~\eqref{eq:16}}{\!\!=\!\!}
    \sum\limits_{i\in I} \psi^n(\varphi^i(p) \!\cdot\! p_i)
    \!\!=\!\!
          \psi^n(p)\,.
  \end{array}
  \]
\end{proof}
\subsection{Derivations on the projective resolution}
\label{sec:deriv-homot}

Let $\partial \colon S \to S$ be an $R$-linear derivation. It defines
an $R$-linear derivation on $S^e$ denoted by $\partial^e$,
\[
\begin{array}{crcl}
  \partial^e \colon & S^e & \to & S^e \\
                    & s \otimes t & \mapsto & \partial(s) \otimes t + s
                                              \otimes \partial(t)\,.
\end{array}
\]
For every left $S^e$-module $M$, a \emph{derivation} of $M$ relative
to $\partial$ is an $R$-linear mapping
\[
\partial_M \colon M \to M
\]
such that, for all $m\in M$ and $s,t\in S$,
\[
\partial_M((s \otimes t) \cdot m) = \partial^e(s \otimes t)\cdot m + (s
\otimes t)\cdot \partial_M(m)\,.
\]
A derivation of $P^\bullet$ relative to $\partial$ is a morphism of
complexes of $R$-modules,
\[
\partial^\bullet \colon P^\bullet \to P^\bullet\,,
\]
such that $\partial^n\colon P^n \to P^n$ is a derivation relative to
$\partial$ for all $n$, and such that
$H^0(\partial^\bullet)= \partial$. Note that a morphism of complexes
of 
$R$-modules $\partial^\bullet \colon P^\bullet \to P^\bullet$ such
that $H^0(\partial^\bullet) = \partial$ is a derivation relative to
$\partial$ if and only if
\begin{equation}
  \label{eq:13}
  \left\{
    \begin{array}{rcl}
      \partial^\bullet \circ \lambda_s
      & = & \lambda_{\partial(s)} + \lambda_s \circ \partial^\bullet\,, \\
      \partial^\bullet \circ \rho_s
      & = & \rho_{\partial(s)} + \rho_s \circ \partial^\bullet\,. \\
    \end{array}\right.
\end{equation}

\begin{rem}
For all derivations $\partial_1^\bullet,\partial_2^\bullet \colon
P^\bullet \to P^\bullet$ relative to $\partial$, the difference
\[
\partial_1^\bullet - \partial_2^\bullet \colon P^\bullet \to
P^\bullet
\]
is a null-homotopic morphism of complexes of left $S^e$-modules.
\end{rem}
\begin{lem}
  \label{sec:deriv-proj-resol}
  There exists a mapping, which need not be linear,
  \begin{equation}
    \label{eq:51}
    \begin{array}{rcl}
      \mathrm{Der}_R(S)  &  \to &  \Hom_R(P^\bullet,P^\bullet) \\
      \partial & \mapsto & \partial^\bullet
    \end{array}
  \end{equation}
  such that:
  \begin{enumerate}
  \item For all $\partial \in \mathrm{Der}_R(S)$, the mapping
    $\partial^\bullet$ is a derivation relative to $\partial$.
  \item For all $\partial_1,\partial_2\in \mathrm{Der}_R(S)$ and
    $r\in R$, there
    exist morphisms of graded left $S^e$-modules
    \[
    \ell,\ell' \colon  P^\bullet  \to  P^\bullet[-1]
    \]
    such that
    \begin{equation}
      \label{eq:17}
      \left\{
        \begin{array}{rcl}
          [\partial_1,\partial_2]^\bullet
          - [\partial_1^\bullet,\partial_2^\bullet]
          & = &
                \ell \circ d + d \circ \ell \\
          (\partial_1+r\partial_2)^\bullet - (\partial_1^\bullet
          +r \partial_2^\bullet)
          & = &
                \ell' \circ d + d \circ \ell'\,.
        \end{array}
      \right.
    \end{equation}
  \item For all $s\in S$ and $\partial \in \mathrm{Der}_R(S)$, there
    exists a morphism of graded $R$-modules
    \[
    h\colon P^\bullet \to P^\bullet[-1]
    \]
    such that
    \begin{equation}
      \label{eq:18}
      (s\partial)^\bullet - \lambda_s\circ \partial^\bullet =
      h \circ d + d \circ h
    \end{equation}
    and, for all $p \in P^\bullet$ and $t_1,t_2\in S$,
    \begin{equation}
      \label{eq:19}
      h((t_1 \otimes t_2) \cdot p)
      =
      (t_1 \otimes t_2) \cdot h(p)
      - (t_1 \otimes \partial(t_2)) \cdot k_s(p)\,.
    \end{equation}
  \end{enumerate}
  Recall that $k_s\colon P^\bullet \to P^\bullet[-1]$ is a morphism of
  graded left $S^e$-modules such that $\lambda_s-\rho_s = k_s \circ d
  + d \circ k_s$ (see \eqref{eq:11} and \eqref{eq:12}).
\end{lem}
\begin{proof}
  (1) Let $\partial \in \mathrm{Der}_R(S)$. For convenience, denote
  $\partial$ by $\partial^1 \colon S\to S$. The proof is an induction
  on $n\leqslant 1$. Let $n\leqslant 0$, and assume that a commutative
  diagram is given
  \[
  \xymatrix{
    P^n \ar[r]^{d^n} & P^{n+1} \ar[r] \ar[d]^{\partial^{n+1}} & \cdots \ar[r] & P^0
    \ar[r]^{d^0} \ar[d]^{\partial^0} & P^1 \ar[r] \ar[d]^{\partial^1}& 0 \\
    P^n \ar[r]^{d^n} & P^{n+1} \ar[r] & \cdots \ar[r] & P^0
    \ar[r]^{d^1} & P^1 \ar[r] & 0 \\
  }
  \]
  where $\partial^i \colon P^i \to P^i$ is a derivation relative to
  $\partial$ for all $i\in \{n+1,n+2,\cdots,0\}$. Let
  \[
  ((p_i,\varphi^i))_{i\in I}
  \]
  be a coordinate system of the projective left $S^e$-module $P^n$
  (see the proof in \ref{sec:data-proj-resol}).  Then, for all $i\in
  I$, there exists $o_i'\in P^n$ such that
  \[
  \partial^{n+1}\circ d^n(p_i) = d^n(p_i')\,.
  \]
  Denote by $\partial^n$
  the $R$-linear mapping from $P^n$ to $P^n$ such that, for all
  $p\in P^n$,
  \[
  \partial^n(p)
  =
  \sum_{i\in I} \partial(\varphi^i(p)) \cdot p_i
  + \varphi^i(p)\cdot p_i'\,.
  \]
  Then, for all $p\in P^n$,
  \[
  \begin{array}{rcl}
    d^n \circ \partial^n(p)
    & = &
          \sum_{i\in I}
          \partial(\varphi^i(p)) \cdot d^n(p_i)
          + \varphi^i(p) \cdot d^n(p_i') \\ \\
    & = &
          \sum_{i\in I}
          \partial(\varphi^i(p)) \cdot d^n(p_i)
          + \varphi^i(p) \cdot \partial^{n+1} \circ d^n(p_i) \\ \\
    & = &
          \partial^{n+1} \circ d^n(\sum_{i\in I} \varphi^i(p) \cdot
          p_i) \\ \\
    & = &
          \partial^{n+1} \circ d^n(p)\,.
  \end{array}
  \]
  Thus,
  \[
  d^n \circ \partial^n = \partial^{n+1} \circ d^n\,.
  \]
  Moreover, $\partial^n$ is a derivation of $P^n$ relative to
  $\partial$ because $\partial$ is a derivation of $S^e$ and
  $\varphi^i \in \Hom_{S^e}(P^n,S^e)$ for all $i\in I$.

  \medskip

  (2) Note that $[\partial_1,\partial_2]^\bullet$ and
  $[\partial_1^\bullet,\partial_2^\bullet]$ (or,
  $(\partial_1+r\partial_2)^\bullet$ and
  $\partial_1^\bullet+r\partial_2^\bullet$) are derivations of
  $P^\bullet$ relative to $[\partial_1,\partial_2]$ (or, to
  $\partial_1+r\partial_2$, respectively). The conclusion therefore
  follows from the remark preceding Lemma~\ref{sec:deriv-proj-resol}.

  \medskip

  (3)  Denote by $\psi$ the mapping
  $(s\partial)^\bullet - \lambda_s\circ \partial^\bullet$ given by
  \[
  \begin{array}{rcl}
    P^\bullet & \to & P^\bullet \\
    p & \mapsto & (s\partial)^\bullet(p) - (s \otimes 1)
                  \cdot \partial^\bullet (p)\,.
  \end{array}
  \]
  Then, for all $p\in P^\bullet$ and $t\in S$,
  \[
  \begin{array}{l}
    \psi((t \otimes 1) \!\cdot\! p) \\ \\
    \!=\!
          (s\partial)^\bullet((t \otimes 1) \!\cdot\! p)
          - (s \otimes 1) \!\cdot\! \partial^\bullet((t \otimes 1)
          \!\cdot\! p) \\ \\
    \!=\!
          (s\partial(t) \otimes 1) \!\cdot\! p + (t \otimes 1)
          \!\cdot\! (s\partial)^\bullet(p)
    - (s \otimes 1) \!\cdot\! (\partial(t) \otimes 1) \!\cdot\! p
    - (s \otimes 1) \!\cdot\! (t \otimes 1)
          \!\cdot\! \partial^\bullet(p) \\ \\
    \!=\!
          (t \otimes 1) \!\cdot\! \psi(p)
  \end{array}
  \]
  and
  \[
  \begin{array}{l}
    \psi((1 \otimes t) \!\cdot\! p) \\ \\
    \underset{\ \ \ \ \ \ }{=}
    \!\!\!\!(s\partial)^\bullet((1 \otimes t) \!\cdot\! p)
          - (s \otimes 1) \!\cdot\! \partial^\bullet((1 \otimes t)
          \!\cdot\! p) \\ \\
    \underset{\ \ \ \ \ \ }{=}
    \!\!\!\!(1 \otimes s\partial(t)) \!\cdot\! p + (1 \otimes t)
          \!\cdot\! (s \partial)^\bullet(p)
          - (s \otimes 1) \!\cdot\! (1 \otimes \partial(t)) \!\cdot\! p
          - (s \otimes 1) \!\cdot\! (1 \otimes t)
          \!\cdot\! \partial^\bullet(p) \\ \\
    \underset{\ \ \ \ \ \ }{=}
    \!\!\!\!(1 \otimes t) \!\cdot\! \psi(p) + (1 \otimes \partial(t)) \!\cdot\!
          (\rho_s - \lambda_s)(p) \\ \\
    \underset{~\eqref{eq:12}}{=}
    \!\!\!\!(1 \otimes t) \!\cdot\! \psi(p) - (1 \otimes \partial(t))\!\cdot\!
          (k_s\circ d + d \circ k_s)(p) \,.
  \end{array}
  \]
  Hence, Lemma~\ref{sec:data-proj-resol-1} may be applied, which yields (3).
\end{proof}

\begin{rem}
Using the remark preceding Lemma~\ref{sec:deriv-proj-resol}, it may be
checked that, although the mapping
$\mathrm{Der}_R(S) \to \Hom_R(P^\bullet,P^\bullet)$ of the lemma is
not unique, two such mappings induce the same mapping from
$\mathrm{Der}_R(S)$ to $H^0\Hom_R(P^\bullet,P^\bullet)$, which is
$R$-linear.
\end{rem}

When $S$ is projective in $\Mod(R)$, it is possible to be more
explicit on a possible mapping
$\partial\mapsto \partial^\bullet$. Indeed, the Hochschild complex
$B(S) = S^{\otimes \bullet+2}$ is a projective resolution of $S$.  For
all $\partial \in \mathrm{Der}_R(S)$, define the following mapping:
\[
\begin{array}{crcl}
  L_\partial \colon & B(S) & \to & B(S) \\
                    & (s_0 | \cdots | s_{n+1}) & \mapsto & \sum\limits_{i=0}^{n+1} (s_0 |
                                                           \cdots
                                                           |s_{i-1}| \partial(s_i)|\cdots
                                                           |s_{i+1} | \cdots | s_n)\,.
\end{array}
\]
This is a derivation of $B(S)$ relative to $\partial$. It is direct to
check that the mapping
\[
\begin{array}{rcl}
  \mathrm{Der}_R(S) & \to & \Hom_R(B(S),B(S)) \\
  \partial & \mapsto & L_\partial
\end{array}
\]
is a morphism of Lie algebras over $R$.
Now, consider homotopy equivalences of complexes of $S^e$-modules,
\[
\xymatrix{
  P^\bullet \ar@<+2pt>[r]^f & B(S)\,, \ar@<+2pt>[l]^g
}
\]
and, for all $\partial \in \mathrm{Der}_R(S)$, define
$\partial^\bullet$ as
\[
\partial^\bullet = g \circ L_\partial \circ f\,;
\]
this is a derivation relative to $\partial$ because so is
$L_\partial$ and because $f$ and $g$ are morphisms of
resolutions of $S$ in $\Mod(S^e)$. The following mapping satisfies
the conclusion of the preceding lemma, it is moreover $R$-linear.
\[
\begin{array}{rcl}
  \mathrm{Der}_R(S) & \to & \Hom_R(P^\bullet,P^\bullet) \\
  \partial & \mapsto & \partial^\bullet\,.
\end{array}
\]

\subsection{Lie derivatives}
\label{sec:lie-derivatives}

Consider a mapping $\partial \mapsto \partial^\bullet$ such as in
Lemma~\ref{sec:deriv-proj-resol}. Let $M$ be an $S$-bimodule and
$\partial\colon S\to S$ be an $R$-linear derivation. Let
$\partial_M \colon M \to M$ be a derivation relative to $\partial$.
Given $n\in \mathbb N$ and $\psi \in \Hom_{S^e}(P^{-n},M)$, denote by
$\mathcal L_\partial(\psi)$ the  mapping
\begin{equation}
  \label{eq:52}
  \mathcal L_\partial(\psi)  = \partial_M \circ \psi - \psi \circ \partial^{-n}\,.
\end{equation}
This is a morphism in $\Mod(S^e)$ because so is $\psi$ and because
$\partial_M$ and $\partial^{-n}$ are derivations relative to
$\partial$; moreover, it is a cocycle (or a coboundary) as soon as
$\psi$ is because $\partial^\bullet \colon P^\bullet \to P^\bullet$ is
a morphism of complexes. Denote by $\mathcal L_\partial$ the resulting
mapping in cohomology
\[
\mathcal L_\partial \colon \Ext^\bullet_{S^e}(S,M) \to \Ext^\bullet_{S^e}(S,M)
\]
such that for all $c\in \Ext^\bullet_{S^e}(S,M)$, say represented by a cocycle
$\psi$, then $\mathcal L_\partial(c)$ is represented by the cocycle
$\mathcal L_\partial(\psi)$. In the situations considered later in the
article, there is no ambiguity on $\partial_M$, whence its omission in
the notation. 

Following similar considerations denote also by $\mathcal L_\partial$
the mapping
\[
\mathcal L_\partial \colon \mathrm{Tor}_\bullet^{S^e}(S,M) \to \mathrm{Tor}_\bullet^{S^e}(S,M)
\]
such that for all $\omega \in \mathrm{Tor}_\bullet^{S^e}(S,M)$, say represented by a
cocycle $m\otimes p\in M \otimes_{S^e}P^\bullet$ with sum sign
omitted, $\mathcal L_\partial(\omega)$ is represented by the cocycle
\[
\mathcal L_\partial(m \otimes p) :=m \otimes \partial^\bullet(p) + 
\partial_M(m) \otimes p\,.
\]

\medskip

When $S$ is projective in $\Mod(R)$, these operations may be written
explicitly in terms of the Hochschild resolution. When $\psi$ is a
Hochschild cocycle lying in $\Hom_R(S^{\otimes n},M)$, the mapping
$\mathcal L_\partial(\psi)$ is given by
\begin{equation}
  \label{eq:21}
  (s_1|\cdots |s_n) \mapsto \partial_M(f(s_1|\cdots|s_n)) -
  \sum\limits_{i=1}^n f(s_1|\cdots|\partial(s_i)|\cdots|s_n)\,.
\end{equation}
Likewise, the operation in Hochschild homology is induced by the
following mapping at the level of Hochschild chains,
\[
\begin{array}{rcl}
  M \otimes S^{\otimes n} & \to & M \otimes S^{\otimes n} \\
  (m|s_1|\cdots|s_n) & \mapsto & (\partial_M(m)|s_1|\cdots |s_n)
                                 + \sum\limits_{i=1}^n (m|s_1|\cdots
                                 | \partial(s_i)| \cdots |s_n)\,.
\end{array}
\]

\medskip

The operator $\mathcal L_\partial$ is of course called the \emph{Lie
  derivative} of $\partial$. When $M=S$ and $S$ is projective in
$\Mod(R)$, this is nothing else but the classical Lie derivative
defined in \cite[Section 9]{MR0154906}. In view of the remark
following Lemma~\ref{sec:deriv-proj-resol}, these constructions
depend only on $\partial$ and $\partial_M$ and not on the choices of
$P^\bullet$ and the mapping $\partial \mapsto \partial^\bullet$.

In the sequel these constructions are considered mainly in the
following cases:
\begin{itemize}
\item $M=S$ and $\partial_M=\partial$.
\item $M=S^e$ and $\partial_M=\partial^e$.
\item $M=\Ext^n_{S^e}(S,S^e)$ ($n\in \mathbb N$) and $\partial_M=\mathcal
  L_\partial$, which makes sense according to the result below.
\end{itemize}
In the sequel the following construction is also used. Consider
$S$-bimodules $M,N$. Let $m,n\in \mathbb N$. Let
$\partial \in \mathrm{Der}_R(S)$ and let $\partial_M \colon M \to M$ and
$\partial_N \colon N\to N$ be $R$-linear derivations relative to
$\partial$. Then, for all $f\in \Hom_R(\Ext^m_{S^e}(S,M),\mathrm{Tor}_n^{S^e}(S,N))$, define
$\mathcal L_\partial(f)$ as
\[
\mathcal L_\partial \circ f - f \circ \mathcal L_\partial\,.
\]

\medskip

Recall that for all $M\in \Mod(S^e)$, the spaces $\Ext^\bullet_{S^e}(S,M)$
and $\mathrm{Tor}_\bullet^{S^e}(S,M)$ are left $S$-modules by means
of $\lambda_s\colon M \to M$ for all $s\in S$; the corresponding
multiplication by $s$ on these (co)homology spaces is denoted by
$\lambda_s$.
\begin{lem}
  \label{sec:lie-derivatives-1}
  Let $M\in \Mod(S^e)$, $n\in \mathbb N$ and $s\in S$. Let
  $\partial,\partial'\colon S\to S$ be $R$-linear derivations. Let
  $\partial_M,\partial'_M\colon M \to M$ be $R$-linear derivations
  relative to $\partial$ and $\partial'$, respectively. Then, the
  following hold in $\Ext^\bullet_{S^e}(S,M)$:
  \begin{enumerate}
  \item $\mathcal L_\partial \circ \lambda_s = \lambda_{\partial(s)} +
    \lambda_s \circ \mathcal L_\partial$.
  \item $\mathcal L_{[\partial,\partial']} = [\mathcal L_\partial,
    \mathcal L_{\partial'}]$.
  \item let $m\in \mathbb N$, let $N$ be another $S$-bimodule and let
    $\partial_N \colon N \to N$ be a derivation relative to
    $\partial$. Consider the contraction mapping
    \[
    \begin{array}{rcl}
      \mathrm{Tor}_m^{S^e}(S,M) & \to & \Hom_R( \Ext^n_{S^e}(S,N) , \mathrm{Tor}_{m-n}^{S^e}(S,M \otimes_S N)) \\
      \omega & \mapsto & (c \mapsto \iota_c(\omega))\,.
    \end{array}
    \]
    If $m=n$, then it is $\mathcal L_\partial$-equivariant. When $S$
    is projective in $\Mod(R)$, it is $\mathcal L_\partial$
    equivariant for all $m,n\in \mathbb N$;
  \item If $M$ is symmetric as an $S$-bimodule, $\mathcal
    L_{s\partial} = \lambda_s\circ \mathcal L_\partial$.
  \item When $M=S^e$ (and $\partial_M=\partial^e$), the following
  equality holds in $\Ext^\bullet_{S^e}(S,M)$:
  \[
  \mathcal L_{s\partial} = \lambda_s \circ \mathcal L_\partial -
  \lambda_{\partial(s)}\,.
  \]
  \end{enumerate}
\end{lem}
\begin{proof}
  (1) The equality is checked on cochains. Let $\psi \in
  \Hom_{S^e}(P^{-n},M)$. Then,
  \[
  \begin{array}{rcl}
    \mathcal L_\partial \circ \lambda_s(\psi)
    & = &
          \partial_M \circ \lambda_s \circ \psi
          - \lambda_s \circ \psi \circ \partial^\bullet \\
    & = &
          (\lambda_{\partial(s)} + \lambda_s \circ \partial_M) \circ
          \psi
          - \lambda_s \circ \psi \circ \partial^\bullet \\
    & = &
          (\lambda_{\partial(s)} + \lambda_s \circ \mathcal
          L_\partial) (\psi)\,.
  \end{array}
  \]

  \medskip
  
  (2) Note that $\mathcal L_{[\partial,\partial']}$ is defined with
  respect to $[\partial_M,\partial_M']$, which is a derivation of $M$
  relative to $[\partial,\partial']$. Following Lemma~\ref{sec:deriv-proj-resol},
  there exists a morphism of graded $S^e$-modules
  \[
  \ell \colon P^\bullet \to P^\bullet[-1]
  \]
  such that
  \[
  [\partial,\partial']^\bullet - [\partial^\bullet,\partial'^\bullet]
  = \ell \circ d+ d\circ \ell\,.
  \]
  Let $\psi  \in \Hom_{S^e}(P^{-n},M)$. If this is a cocycle, then
  \[
  \begin{array}{rcl}
    \mathcal L_{[\partial,\partial']}(\psi)
    & = &
          [\partial_M,\partial_M'] \circ \psi
          - \psi \circ ([\partial^\bullet,\partial'^\bullet] + \ell
          \circ d + d \circ \ell) \\
    & = &
          [\mathcal L_\partial, \mathcal L_{\partial'}](\psi)
          - \psi \circ \ell \circ d
          - \underset{=0}{\underbrace{\psi \circ d}} \circ \ell\,,
  \end{array}
  \]
  which is cohomologous to $[\mathcal L_\partial \mathcal
  L_{\partial'}](\psi)$. This proves (2).
  
  \medskip

  (3) Note that the mapping
  \[
  \begin{array}{rrcl}
    \partial_{M\otimes_SN} \colon
    & M \otimes_SN & \to & M \otimes_S N \\
    & x \otimes y & \mapsto & \partial_M(x) \otimes y + x
                              \otimes \partial_N(y)\,,
  \end{array}
  \]
  is a well-defined derivation relative to $\partial$, which defines
  $\mathcal L_\partial$ on
  \[
  \mathrm{Tor}_{m-n}^{S^e}(S,M\otimes_S N)\,.
  \]

  Assume first that $m=n$. Let $p^0$ be any element of the preimage of
  $1_S$ under the augmentation mapping $P^0\to S$. Let
  $x \otimes p \in M \otimes_{S^e}P^{-m}$ and
  $\psi \in \Hom_{S^e}(P^{-m},N)$, and use the notation
  \[
  \iota_\psi(x \otimes p) := (x\otimes \psi(p)) \otimes
  p^0\,.
  \]
  Recall that the contraction mapping is
  induced by the mapping
  \[
  \begin{array}{rcl}
    M \otimes_{S^e}P^{-m}
    & \to &
            \Hom_R(\Hom_{S^e}(P^{-m},N) , (M\otimes_S
            N)\otimes_{S^e}P^0) \\
    x \otimes p & \mapsto & \iota_?(x \otimes p)
  \end{array}
  \]
  Denote
  $\mathcal L_\partial(\iota_\psi(x \otimes p))
  - \iota_{\mathcal L_\partial(\psi)}(x \otimes p)$ by $\delta$. Then,
  \[
  \begin{array}{rcl}
    \delta
    & = &
          \mathcal L_\partial( (x \otimes \psi(p)) \otimes p^0)
          - (x \otimes \mathcal L_\partial(\psi)(p)) \otimes p^0 \\ \\
    & = &
          \partial_M(x) \otimes \psi(p) \otimes p^0 
          + x \otimes \partial_N(\psi(p)) \otimes p^0
          + x \otimes \psi(p) \otimes \partial^0(p^0) \\
    &   &
    \hfill      - x \otimes \partial_N(\psi(p)) \otimes p^0
          + x \otimes
          \psi(\partial^{-m}(p)) \otimes p^0 \\ \\
    & = &
          \iota_\psi(\mathcal L_\partial(x \otimes p)) + x \otimes
          \psi(p) \otimes \partial^0(p^0)\,.
  \end{array}
  \]
  Note that $\partial^0(p_0)$ lies in the image of
  $d\colon P^{-1}\to P^0$ because the image of $p^0$ under $P^0 \to S$
  is $1$ and $H^0(\partial^\bullet) = \partial$. These considerations
  therefore prove (3) when $m=n$.

  \medskip

  Now assume that $S$ is projective in $\Mod(R)$. Then, the
  equivariance may be checked at the level of Hochschild
  (co)chains. Let $o=(x|s_1|\cdots|s_m) \in S^{\otimes m}$ and
  $\psi \in \Hom_R(S^{\otimes n },N)$. Then,
  \[
  \begin{array}{l}
    \mathcal L_\partial(\iota_\psi (o)) - \iota_{\mathcal L_\partial
    (\psi)}(o) \\ \\
    \!\!=\!\!
          \mathcal L_\partial(
          x \otimes \psi(s_1|\cdots | s_n)|s_{n+1}|\cdots |s_m)
          - (x \otimes \mathcal L_\partial(\psi)(s_1|\cdots | s_n)
          | s_{n+1}|\cdots |s_m) \\ \\
    \!\!=\!\!
          (\partial_M(x) \otimes \psi(s_1|\cdots|s_n)
          |s_{n+1}|\cdots |s_m) 
          + (x \otimes \partial_N(\psi(s_1|\cdots|s_n))
          |s_{n+1}|\cdots |s_m) \\ \\
    \null \hfill + \sum\limits_{j=n+1}^m
          (x \otimes \psi(s_1|\cdots | s_n)|s_{n+1}|\cdots
    |\partial(s_j)|\cdots |s_m)
    \hfill \null
    \\ \\
    \null \hfill
    - (x \otimes \partial_N(\psi(s_1|\cdots |s_n))|s_{n+1}|\cdots
    |s_m)
    \hfill \null \\ \\
    \null \hfill + \sum\limits_{j=1}^n
          (x \otimes \psi(s_1|\cdots|\partial(s_j)|\cdots |s_n)|s_{n+1}|\cdots
          |s_m) \\ \\
    \!\!=\!\!
          \iota_\psi(\mathcal L_\partial(o))\,,
  \end{array}
  \]
  which proves (3) for all $m,n\in \mathbb N$ when $S$ is projective
  in $\Mod(R)$.

  \medskip

  (4) Note that $\mathcal L_{s\partial}$ is defined with respect to
  the derivation $s\partial_M$ ($=\lambda_s\circ \partial_M$). Assume that $M$
  is symmetric as an $S$-bimodule. Therefore, the mapping
  \[
  \lambda_s \circ \partial^\bullet \colon P^\bullet \to P^\bullet
  \]
  is a derivation relative to $s\partial$. Let $\psi \in
  \Hom_{S^e}(P^\bullet,M)$ be a cocycle with cohomology class denoted
  by $c$. Since $\psi\circ \lambda_s
  = \lambda_s \circ \psi$,
  \[
  \mathcal L_{s\partial}(\psi)
  =
  (\lambda_s\circ \partial_M) \circ \psi - \psi \circ (\lambda_s
  \circ \partial^\bullet) = \lambda_s \circ \mathcal L_\partial(\psi)\,.
  \]
  Taking cohomology classes yields that $\mathcal L_{s\partial}(c) =
  \lambda_s\circ \mathcal L_\partial(c)$.

  \medskip

  (5) Recall that, here, $\partial_M$ is taken equal to
  \[
  \begin{array}{crcl}
    (s\partial)^e \colon & S^e & \to & S^e \\
                         & s_1 \otimes s_2 & \mapsto &
                                                       s\partial(s_1)
                                                       \otimes s_2
                                                       + s_1 \otimes s\partial(s_2)\,.
  \end{array}
  \]
  Let $\psi \in \Hom_{S^e}(P^{-n},M)$ be a cocycle with cohomology
  class denoted by $c$.
  Let $h$ be as in part (3) of Lemma~\ref{sec:deriv-proj-resol}.
  Then,
  \[
  \begin{array}{rcl}
    \mathcal L_{s\partial}(\psi)
    & = &
          (s\partial)^e\circ \psi - \psi \circ (s\partial)^\bullet \\
    & = &
          (s\partial \otimes 1 + 1 \otimes s\partial) \circ \psi -
          \psi \circ (s \partial)^\bullet \\ 
    & = &
          \lambda_s \circ (\partial \otimes 1) \circ \psi
          + \rho_s \circ ( 1 \otimes \partial) \circ \psi
          - \psi \circ (s \partial)^\bullet\,.
  \end{array}
  \]
  Using \eqref{eq:18}, the equality becomes
  \[
    \mathcal L_{s\partial}(\psi)
    \!\!=\!\!
          \lambda_s \circ (\partial \otimes 1) \circ \psi
          + \rho_s \circ (1 \otimes \partial) \circ \psi \\
          - \lambda_s \circ \psi \circ \partial^\bullet - \psi
          \circ h \circ d - \underset{=0}{\underbrace{\psi \circ d}}
          \circ h\,.
  \]
  Using $[\partial,\rho_s] = \rho_{\partial(s)}$, it then becomes
  \[
  \begin{array}{rl}
    \mathcal L_{s\partial}(\psi)
    & \!=\! 
          \lambda_s \circ (\partial \otimes 1) \circ \psi + (1
          \otimes \partial) \circ \rho_s \circ \psi
          - \rho_{\partial(s)} \circ \psi 
          - \lambda_s \circ \psi \circ \partial^\bullet - \psi
          \circ h \circ d \\
    & \!=\! 
          \lambda_s \circ (\partial \otimes 1) \circ \psi
          - \rho_{\partial(s)} \circ \psi
          + (1 \otimes \partial) \circ \psi \circ (\rho_s -
          \lambda_s) \\
    & \hfill + (1 \otimes \partial) \circ \psi \circ \lambda_s
          - \lambda_s \circ \psi \circ \partial^\bullet
          - \psi \circ h \circ d\,.
  \end{array}
  \]
  Using \eqref{eq:12}, this becomes
  \[
  \begin{array}{rl}
    \mathcal L_{s\partial}(\psi)
    & \!\!= \!\!
          \lambda_s \!\circ\! (\partial \!\otimes\! 1) \!\circ\! \psi
          - \rho_{\partial(s)} \!\circ\! \psi
          - (1 \!\otimes\! \partial) \!\circ\!
          \underset{=0}{\underbrace{\psi \!\circ\! d}} \!\circ\! k_s
    \\ \\
    & \hfill
          - (1 \!\otimes\! \partial) \!\circ\! \psi \!\circ\! k_s \!\circ\! d
          + \underset{\scriptstyle
          =(1 \otimes \partial) \circ \lambda_s \circ \psi =
          \lambda_s \circ (1 \otimes \partial) \circ
          \psi}{\underbrace{(1 \!\otimes\! \partial) \!\circ\! \psi
          \!\circ\! 
          \lambda_s}} 
          - \lambda_s \!\circ\! \psi \!\circ\! \partial^\bullet
          - \psi \!\circ\! h \!\circ\! d \\ \\
    & \!\!= \!\!
          \lambda_s \!\circ\! (\partial \!\otimes\! 1 + 1
          \!\otimes\! \partial) \!\circ\! \psi
          - \rho_{\partial(s)} \!\circ\! \psi
          - \lambda_s \!\circ\! \psi \!\circ\! \partial^\bullet
          -(\psi \!\circ\! h + (1 \!\otimes\! \partial) \!\circ\! \psi
          \!\circ\! k_s) \!\circ\! d \\ \\
    & \!\!= \!\!
          \lambda_s \!\circ\! (\mathcal L_\partial(\psi)) - \rho_{\partial(s)} \!\circ\!
          \psi
          - (\psi \!\circ\! h + (1 \!\otimes\! \partial) \!\circ\! \psi
          \!\circ\! k_s) \!\circ\! d\,.
  \end{array}
  \]
  Now, consider the following $R$-linear mapping denoted by $f$:
  \[
  \psi \circ h + (1 \otimes \partial) \circ \psi \circ k_s
  \colon P^{-n+1} \to S^e\,.
  \]
  This is a morphism of $S$-bimodules. Indeed,
  \begin{itemize}
  \item it is a morphism of left $S$-modules because so are $\psi$, $1
    \otimes \partial$, $k_s$ and $h$ (see \eqref{eq:19});
  \item since $\psi$ and $k_s$ are morphisms of $S$-bimodules,
    then, for all $t\in S$,
    \[
    \begin{array}{rcl}
      f \circ \rho_t
      & = &
            \psi \circ h \circ \rho_t
            + (1 \otimes \partial) \circ \rho_t \circ \psi
            \circ k_s \\ \\
      & \underset{~\eqref{eq:19}}{=}
          &
            \psi \circ (\rho_t \circ h - \rho_{\partial(t)} \circ k_s)
            + (1 \otimes \partial) \circ \rho_t \circ \psi
            \circ k_s \\ \\
      & = &
            \rho_t \circ \psi \circ h - \rho_{\partial(t)} \circ \psi
            \circ k_s + (1 \otimes \partial) \circ \rho_t \circ
            \psi \circ k_s \\ \\
      & = &
            \rho_t \circ \psi \circ h + \rho_t \circ (1
            \otimes \partial) \circ \psi \circ k_s \\ \\
      & = &
            \rho_t \circ f\,.
    \end{array}
    \]
  \end{itemize}
  Therefore, $\mathcal L_{s\partial}(\psi)$ and $\lambda_s \circ
  \mathcal L_\partial(\psi) - \rho_{\partial(s)}
  \circ \psi$ are cohomologous. Since so are $\lambda_{\partial(s)}
  \circ \psi$ and $\rho_{\partial(s)} \circ \psi$ it follows that
  \[
  \mathcal L_{s\partial}(c) = \lambda_s \circ\mathcal L_\partial(c)
  - \lambda_{\partial(s)}(c)\,.
  \]
\end{proof}
\subsection{The action of $L$ on $\Ext^\bullet_{S^e}(S,S^e)$}
\label{sec:acti-l-extb}

According to Lemma~\ref{sec:lie-derivatives-1}, the mapping 
\begin{equation}
  \label{eq:20}
  \begin{array}{ccc}
    L \times \Ext^n_{S^e}(S,S^e) & \to & \Ext^n_{S^e}(S,S^e) \\
    (\alpha, e) & \mapsto & \alpha \cdot e := \mathcal L_{\partial_\alpha}(e)
  \end{array}
\end{equation}
endows $\Ext^\bullet_{S^e}(S,S^e)$ with a compatible left
$S\rtimes L$-module structure in the sense of \eqref{eq:4}, that is, a
left $S\rtimes L$-module structure such that, for all
$e\in \Ext^\bullet_{S^e}(S,S^e)$, $\alpha\in L$ and $s\in S$,
\begin{equation}
  \label{eq:53}
  (s\alpha) \cdot e = s \cdot (\alpha \cdot e) - \alpha(s)
  \cdot e\,.
\end{equation}
This left $S\rtimes L$-module structure on $\Ext^\bullet_{S^e}(S,S^e)$
does not define a left $U$-module structure in general. However,
Lemma~\ref{sec:left-right-u} yields that
$\Ext^\bullet_{S^e}(S,S^e)^\vee$ is a right $U$-module by defining
$\theta\cdot \alpha$, for all
$\theta \in \Ext^\bullet_{S^e}(S,S^e)^\vee$ and $\alpha \in L$, as
\[
\begin{array}{crcl}
  \theta \cdot \alpha \colon &\Ext^n_{S^e}(S,S^e) & \to & S \\
                             &e & \mapsto & - \alpha(\theta(e)) +
                                            \theta(\alpha \cdot e)
                                            \,.
\end{array}
\]

\subsection{Particular case of Van den Bergh and Calabi-Yau duality}
\label{sec:part-case-calabi}

Recall that, whenever $\mathrm{Tor}^{S^e}_n(S,S)\simeq S$ as $S$-(bi)modules, a
\emph{volume form} is a free generator $\omega_S$ of $\mathrm{Tor}^{S^e}_n(S,S)$;
and the associated \emph{divergence}
\[
\mathrm{div} \colon \mathrm{Der}_R(S) \to S
\]
is defined such that, for all $\partial \in \mathrm{Der}_R(S)$,
\begin{equation}
  \label{eq:55}
  \mathcal L_\partial(\omega_S) = \mathrm{div}(\partial)\omega_S\,.
\end{equation}

When $S$ is Calabi-Yau in dimension $n$, any free generator
$e_S$ of the left $S$-module $\Ext^n_{S^e}(S,S^e)$ defines an isomorphism of
$S$-bimodules
\[
\begin{array}{rrcl}
  \theta \colon & S & \to & \Ext^n_{S^e}(S,S^e) \\
                & s & \mapsto & s e_S\,.
\end{array}
\]
In such a situation, the fundamental class
$c_S \in \mathrm{Tor}_n^{S^e}(S,\Ext^n_{S^e}(S,S^e))$ (see
Section~\ref{sec:fund-class-contr}) is a free generator of the left
$S$-module $\mathrm{Tor}_n^{S^e}(S,\Ext^n_{S^e}(S,S^e))$; and hence
the preimage $\omega_S$ of $c_S$ under the bijective mapping
\[
\theta_* \colon \mathrm{Tor}_n^{S^e}(S,S) \to \mathrm{Tor}_n^{S^e}(S,\Ext^n_{S^e}(S,S^e))\,.
\]
is a volume form for $S$, thus defining a divergence operator.
\begin{prop}
  \label{sec:particular-case-van}
  The following properties hold.
  \begin{enumerate}
  \item Assume the following:
    \begin{itemize}
    \item $R$ is Noetherian and $S$ is finitely generated as an
      $R$-algebra.
    \item $S$ is projective in $\Mod(R)$.
    \item $S$ has Van den Bergh duality with dimension $n$.
    \end{itemize}
    Then there is an isomorphism of $S$-modules compatible with Lie
    derivatives
    \[
    \Ext^n_{S^e}(S,S^e) \simeq \Lambda_S^n\mathrm{Der}_R(S)\,.
    \]
  \item  Assume that $S$ is Calabi-Yau in dimension $n$. Let $e_S$ be a
    free generator of the left $S$-module $\Ext^n_{S^e}(S,S^e)$. Let $\mathrm{div}$
    be the resulting divergence operator. Then, for all $\partial \in
    \mathrm{Der}_R(S)$,
    \begin{equation}
      \label{eq:56}
      \mathcal L_{\partial}(e_S) = -\mathrm{div}(\partial) e_S\,.
    \end{equation}
  \end{enumerate}
\end{prop}
\begin{proof}
  In both cases, $S$ lies in $\mathrm{per}(S^e)$. Denote the fundamental
  class of $S$ by $c_S$.  In view of part (3) of
  Lemma~\ref{sec:lie-derivatives-1}, the definition of $c_S$ gives
  that
  \begin{equation}
    \label{eq:60}
    \mathcal L_\partial(c_S) = 0\,.
  \end{equation}
  
  (1) Denote $\Ext^n_{S^e}(S,S^e)$ by $D$. In view of
  Proposition~\ref{sec:relat-regul}, \cite[Theorem 3.1]{MR0142598} applies
  and yields an isomorphism of $S$-modules
  \begin{equation}
    \label{eq:61}
    \mathrm{Tor}_n^{S^e}(S,S) \simeq \Lambda^n_S\Omega_{S/R}\,.
  \end{equation}
  Following \cite[Section 9]{MR0154906}, this isomorphism is
  compatible with Lie derivatives. Identify $D^{-1}$ with
  $\Hom_S(D,S)$ and define $\partial_{D^{-1}}$ as follows, for all
  $\partial \in \mathrm{Der}_R(S)$,
  \[
  \begin{array}{rrcl}
    \partial_{D^{-1}} \colon & \Hom_S(D,S) & \to & \Hom_S(D,S) \\
                             & f & \mapsto & \partial \circ f - f \circ \mathcal L_\partial\,.
  \end{array}
  \]
  The evaluation isomorphism
  \begin{equation}
    \label{eq:62}
    \mathrm{ev} \colon D\otimes_SD^{-1} \xrightarrow{\sim} S
  \end{equation}
  is compatible with Lie derivatives in the following sense, where $\partial \in \mathrm{Der}_R(S)$,
  \begin{equation}
    \label{eq:63}
    \partial \circ \mathrm{ev} = \mathrm{ev} \circ (\mathcal L_\partial
    \otimes \mathrm{Id} + \mathrm{Id} \otimes \partial_{D^{-1}})\,.
  \end{equation}
  Besides, the duality isomorphism
  \begin{equation}
    \label{eq:64}
    \iota_?(c_S) \colon \Ext^0_{S^e}(S,D^{-1}) \to \mathrm{Tor}_n^{S^e}(S,D\otimes_S D^{-1})
  \end{equation}
  is compatible with the action of Lie derivatives because of
  (\ref{eq:60}) (see part (3) of
  Lemma~\ref{sec:lie-derivatives-1}). Combining (\ref{eq:61}),
  (\ref{eq:62}), (\ref{eq:63}) and (\ref{eq:64}) yields an isomorphism
  that is compatible with Lie derivatives
  \[
  D^{-1} \simeq \Lambda^n_S\Omega_{S/R}\,.
  \]
  This proves (1).

  \medskip

  (2) Keep the notation $c_S$, $\omega_S$, $\theta$, $\theta_*$ for
  the objects defined from $e_S$ before the statement of the
  proposition.  Let $\partial \in \mathrm{Der}_R(S)$. There exists
  $\lambda\in S$ such that
  \[
  \mathcal L_\partial(e_S) = \lambda e_S\,.
  \]
  Now, for all $s \otimes p \in S\otimes_{S^e}P^{-n}$, 
  \[
  \begin{array}{rcl}
    \mathcal L_\partial(\theta_*(s \otimes p))
    & = &
          \mathcal L_\partial(se_S\otimes p) \\
    & = &
          \partial(s)e_S\otimes p 
          + s \mathcal L_\partial(e_S) \otimes p 
          + s e_S \otimes \partial^\bullet(p)  \\
    & = &
          \theta_*(\mathcal L_\partial(s \otimes p)) 
          +\lambda\theta_*(s \otimes p)\,.
  \end{array}
  \]
  Therefore,
  \[
  0
  = 
  \mathcal L_\partial(c_S)
  = 
  \mathcal L_\partial(\theta_*(\omega_S))
  = 
  \theta_*(\underset{=\mathrm{div}(\partial) \omega_S}{\underbrace{\mathcal L_\partial(\omega_S)}}) + \lambda \theta_*(\omega_S)
  =
  (\lambda+\mathrm{div}(\partial)) c_S\,.
  \]
  Since $c_S$ is regular, $\lambda = -\mathrm{div}(\partial)$.
\end{proof}

\section{Proof of the main theorems}
\label{sec:proof-main-theorem}

The main results of this article are proved in this section. For this
purpose, a description of $\Ext^\bullet_{U^e}(U,U^e)$ is made in
\ref{sec:proof-theorem}; the underlying $S$-module is expressed in
terms of $\Ext^\bullet_{S^e}(S,S^e)$ and $\Ext^\bullet_U(S,U)$; and
the $U$-bimodule structure is described using the functor
$F\colon \Mod(U) \to \Mod(U^e)$ and the action of $L$ on
$\Ext^\bullet_{S^e}(S,S^e)$ introduced in
Section~\ref{sec:left-u-module}. This description is applied in
\ref{sec:proof-theor-refs} in order to prove
Theorem~\ref{sec:introduction}. And Theorem~\ref{sec:introduction-1}
and Corollary~\ref{sec:main-results-1} are proved in
\ref{sec:proof-theor-refs-1} and \ref{sec:case-poiss-algebr} by
specialising to the situations where $\Ext^\mathrm{top}_{S^e}(S,S^e)$
and $\Ext^\mathrm{top}_U(S,U)$ are free, and where $(S,L)$ arises from
a Poisson bracket on $S$, respectively.

Throughout the section, $\Ext^\bullet_{S^e}(S,S^e)$ is endowed with
its compatible left $S\rtimes L$-module structure introduced in
\ref{sec:acti-l-extb}.

\subsection{The inverse dualising bimodule of $U$}
\label{sec:proof-theorem}

This subsection proves the following result.
\begin{prop}
  \label{sec:inverse-dual-bimod}
  Let $R$ be a commutative ring and $d\in \mathbb N$. Let $(S,L)$ be
  a Lie-Rinehart algebra over $R$. 
    Assume the following:
  \begin{enumerate}[(a)]
  \item $S$ is flat as an $R$-module.
  \item for all $n\in \mathbb N$, the $S$-module $\Ext^n_{S^e}(S,S^e)$
    is projective.
  \item $S\in \mathrm{per}(S^e)$.
  \item $L$ is finitely generated and projective with constant rank
    equal to $d$ in $\Mod(S)$.
  \end{enumerate}
  Then,
  $\Lambda_S^dL^\vee \otimes_S \Ext_{S^e}^\bullet (S,S^e)$ is a graded left
  $U$-module such that, for all $\alpha \in L$,
  $c\in \Ext_{S^e}^\bullet(S,S^e)$ and $\varphi \in \Lambda_S^dL^\vee$,
  \[
  \alpha \cdot (\varphi \otimes c) = - \varphi \cdot \alpha \otimes c +
  \varphi \otimes \alpha \cdot c\,.
  \]
  Moreover, $U$ is homologically smooth. Finally, there is an isomorphism of
  graded right $U^e$-modules 
  \[
  \Ext_{U^e}^\bullet(U,U^e) \simeq F( \Lambda^d_SL^\vee \otimes_S
  \Ext_{S^e}^{\bullet-d}(S,S^e)) \,.
  \]
\end{prop}

For this subsection, assume (a), (b), (c) and (d) are true,
and consider
\begin{itemize}
\item a bounded resolution $Q^\bullet \to S$ in $\Mod(U)$ by finitely
  generated and projective modules (see \cite[Lemma 4.1]{MR0154906}),
\item a bounded resolution $\pi \colon P^\bullet \to S$ in
  $\Mod(S^e)$ by finitely generated and projective modules,
\item an injective resolution $j \colon U^e \to I^\bullet$ in
  $\Mod(U^e\otimes (U^e)^\mathrm{op})$.
\end{itemize}
Since $S$ is flat over $R$ and $L$ is projective in $\Mod(S)$,
part (2) of Lemma~\ref{sec:cont-notat-used} gives that $U^e$ is flat
over $R$. Therefore, the extension-of-scalars functor
\[
- \otimes U^e \colon \Mod(U^e) \to \Mod(U^e \otimes (U^e)^\mathrm{op})
\]
is exact. Hence, the
restriction-of-scalars-functor transforms injective $U^e$-bimodules
into injective left $U^e$-modules. Thus, $I^\bullet$ is an
injective resolution of $U^e$ in $\Mod(U^e)$.
Therefore, there is an isomorphism of graded right $U^e$-modules
\begin{equation}
  \label{eq:22}
  \Ext^\bullet_{U^e}(U,U^e) \simeq H^\bullet \Hom_{U^e}(U,I^*)\,.
\end{equation}
The right-hand side is a right $U^e$-module by means of $I^*$.

The proof of the above proposition is divided into separate lemmas.
\begin{lem}
  \label{sec:u-homol-smooth}
  $U$ is homologically smooth.
\end{lem}
\begin{proof}
  Since $U$ is projective in $\Mod(S)$ (see part (2) of
  Lemma~\ref{sec:cont-notat-used}), the functor
  \[
  F \colon \Mod(U) \to \Mod(U^e)
  \]
  is exact. Moreover, $F(S) \simeq U$ and $S\in
  \mathrm{per}(U)$.
  Therefore, in order to prove that $U$ is homologically smooth, it
  suffices to prove that $F(U)\in \mathrm{per}(U^e)$, which is
  equivalent to $F(U)$ being compact in the derived category
  $\mathcal D(U^e)$ of complexes of $U$-bimodules. Here is a proof of
  this fact.  Let $(M_k)_{k\in K}$ be a family in $\mathcal D(U^e)$,
  denote $\oplus_{k\in K} M_k$ by $M$, and consider fibrant
  resolutions of complexes of $U$-bimodules $M_k\to i(M_k)$, for all
  $k\in K$, and $M\to i(M)$. Since $S$ is homologically smooth, then
  $S$ is compact in $\mathcal D(S^e)$, and hence the following natural
  mapping is a quasi-isomorphism,
\[
\bigoplus_{k\in K}\mathrm{Hom}_{S^e}(P^\bullet,M_k) \to
\mathrm{Hom}_{S^e}(P^\bullet,M)\,.
\]
Since $P^\bullet$ is a right bounded complex of projective $S$-bimodules, then
the functor $\mathrm{Hom}_{S^e}(P^\bullet,-)$ preserves
quasi-isomorphisms, and hence the following natural mapping is a
quasi-isomorphism,
\[
\bigoplus_{k\in K}\mathrm{Hom}_{S^e}(P^\bullet,i(M_k)) \to
\mathrm{Hom}_{S^e}(P^\bullet,i(M))\,.
\]
Since $U$ is projective over $S$ on both sides, $U^e$ is
projective in $\mathrm{Mod}(S^e)$. Therefore, for all fibrant
complexes $I$ of $U$-bimodules, the functor $\mathrm{Hom}_{S^e}(-,I)$
preserves quasi-isomorphisms. Accordingly, the following natural
mapping is a quasi-isomorphism:
\[
\bigoplus_{k\in K}\mathrm{Hom}_{S^e}(S,i(M_k)) \to \mathrm{Hom}_{S^e}(S,i(M))\,.
\]
Since the pair $(F,G)$ is adjoint and $G$ is induced by the functor
$\mathrm{Hom}_{S^e}(S,-)$, the following natural mapping is a
quasi-isomorphism:
\[
\bigoplus_{k\in K}\mathrm{Hom}_{U^e}(F(U),i(M_k)) \to
\mathrm{Hom}_{U^e}(F(U),i(M))\,.
\]
Taking cohomology in degree $0$ yields that the following natural
mapping is bijective:
\[
\bigoplus_{k\in K}\mathcal D(U^e)(F(U),i(M_k)) \to
\mathcal D(U^e)(F(U),i(M))\,.
\]
This proves that $F(U)$ is compact in $\mathcal D(U^e)$. Thus, $U$ is
homologically smooth.
\end{proof}

The authors thank Bernhard Keller for having pointed out an incorrect
argument in a previous version of this proof.

\begin{lem}
  \label{sec:isom-extb-ue}
  There is an isomorphism of graded right $U^e$-modules
  \begin{equation}
    \label{eq:26}
    \Ext^\bullet_{U^e}(U,U^e) \simeq H^\bullet(\Hom_U(Q^*,U)
    \otimes_U G(I^*))\,.
  \end{equation}
\end{lem}
\begin{proof}
  Because of the isomorphism $F(S)\simeq U$ in $\Mod(U^e)$ and the
  adjunction $(F,G)$, there is a functorial isomorphism of complexes of
  right $U^e$-modules
  \begin{equation}
    \label{eq:23}
    \Hom_{U^e}(U,I^\bullet) \simeq \Hom_U(S,G(I^\bullet))\,.
  \end{equation}
  Since $F$ is exact and the pair $(F,G)$ is adjoint, $G(I^\bullet)$
  is a left bounded complex of injective left $U$-modules. Hence,
  $\Hom_U(-,G(I^\bullet))$ preserves quasi-isomorphisms. Thus, the
  quasi-isomorphism $Q^\bullet \to S$ induces a quasi-isomorphism of
  complexes of right $U^e$-modules
  \begin{equation}
    \label{eq:24}
    \Hom_U(S,G(I^\bullet)) \to \Hom_U(Q^\bullet,G(I^\bullet))\,.
  \end{equation}
  Since $Q^\bullet$ is bounded and consists of finitely generated
  projective left $U$-modules, the following canonical mapping is a
  functorial isomorphism:
  \begin{equation}
    \label{eq:25}
    \Hom_U(Q^\bullet,U) \otimes_U G(I^\bullet) \to
    \Hom_U(Q^\bullet,G(I^\bullet))\,.
  \end{equation}
  Note that, whether in (\ref{eq:23}), (\ref{eq:24}), or
  (\ref{eq:25}), the involved right $U^e$-module structures are
  inherited from $I^\bullet$. Thus, the announced isomorphism is
  proved.
\end{proof}

In order to examine the right-hand side of \eqref{eq:26} by means of a
spectral sequence, the following lemma describes $H^\bullet(G(I^*))$
as a graded $U-U^e$-bimodule.
\begin{lem}
  \label{sec:cohomology-gibullet}
  Consider $\Ext_{S^e}^\bullet(S,S^e)$ as a left $S \rtimes L$-module
  as in Section~\ref{sec:acti-l-extb}. Then, there is a
  $U-U^e$-bimodule structure on
  $\Ext_{S^e}^\bullet(S,S^e) \otimes_{S^e} U^e$ such that the
  $U^e$-module structure is inherited from $U^e$ and for all
  $\alpha\in L$, $c\in \Ext_{S^e}^\bullet(S,S^e)$ and $u,v\in U$,
  \[
  \alpha \cdot (c \otimes (u \otimes v)) = \alpha \cdot c \otimes (u
  \otimes v) + c \otimes ((\alpha \otimes 1 - 1 \otimes \alpha) \cdot
  (u \otimes v))\,.
  \]
  For this structure, there is an isomorphism of graded
  $U-U^e$-bimodules
  \[
  H^\bullet(G(I^*)) \simeq \Ext_{S^e}^\bullet(S,S^e)
  \otimes_{S^e} U^e\,.
  \]
\end{lem}
\begin{proof}
  The object $G(I^\bullet)$ is $\Hom_{S^e}(S,I^\bullet)$ as a complex
  of $S$-modules, its right $U^e$-module structure is inherited from
  $I^\bullet$, and the one of left $U$-module is given in
  Section~\ref{sec:funct-gcol-modu-1}.

  First, since $U^e$ is projective in $\Mod(S^e)$ and $I^\bullet$
  consists of injective left $U^e$-modules, then $I^\bullet$ is a left
  bounded complex of injective left $S^e$-modules. Hence,
  $\Hom_{S^e}(-,I^\bullet)$ preserves quasi-isomorphisms. Thus,
  $\pi\colon P^\bullet \to S$ induces a quasi-isomorphism of complexes
  of right $S^e$-modules
  \begin{equation}
    \label{eq:36}
    \pi' \colon \Hom_{S^e}(S,I^\bullet) \to \Hom_{S^e}(P^\bullet,
    I^\bullet)\,.
  \end{equation}
  For all $\alpha \in L$, let $\partial_\alpha^\bullet \colon P^\bullet
  \to P^\bullet$ be a derivation relative to $\partial_\alpha \colon
  S \to S$ (see Section~\ref{sec:deriv-homot}), and denote by $\delta_\alpha^\bullet$ the mapping from
  $I^\bullet$ to $I^\bullet$ given by
  \[
  i \mapsto (\alpha \otimes 1 - 1 \otimes \alpha)\cdot i\,.
  \]
  Then, define $\alpha \cdot f$ and $\alpha \cdot g$, for all
  $f \in \Hom_{S^e}(S,I^\bullet)$ and
  $g \in \Hom_{S^e}(P^\bullet,I^\bullet)$, by
  \[
  \begin{array}{rcl}
    \alpha \cdot f
    & = &
          \delta_\alpha^\bullet \circ f - f
          \circ \partial_\alpha \\
    \alpha \cdot g
    & = &
          \delta_\alpha^\bullet \circ g - g
          \circ \partial_\alpha^\bullet\,;
  \end{array}
  \]
  since $\pi \circ \partial_\alpha^\bullet = \partial_\alpha
  \circ \pi$,
  \[
  \pi'(\alpha \cdot f) = \alpha \cdot \pi'(f)\,.
  \]

  The hypotheses on $P^\bullet$ yield  an
  isomorphism of complexes of right $U^e$-modules,
  \begin{equation}
    \label{eq:37}
    \mathrm{ev} \colon \Hom_{S^e}(P^\bullet, S^e) \otimes_{S^e} I^\bullet \to
    \Hom_{S^e}(P^\bullet, I^\bullet)\,.
  \end{equation}
  Endow the left-hand side term with the following action of $L$.  For
  all $\alpha \in L$ and
  $\varphi \otimes i \in \Hom_{S^e}(P^\bullet, S^e) \otimes_{S^e}
  I^\bullet$,
  denote by $\alpha \cdot (\varphi \otimes i)$ the (well-defined)
  element of $\Hom_{S^e}(P^\bullet,S^e) \otimes_{S^e} I^\bullet$,
  \[
  \alpha \cdot \varphi \otimes i + \varphi \otimes (\delta_\alpha^\bullet i)\,.
  \]
  The assignment
  $\varphi \otimes i \mapsto \alpha \cdot (\varphi \otimes i)$ is a
  morphism of complexes of $R$-modules from
  $\Hom_{S^e}(P^\bullet, S^e) \otimes_{S^e} I^\bullet$ to itself. In
  view of (\ref{eq:17}) and of the identity
  \[
  (\alpha \otimes 1 - 1 \otimes \alpha) \cdot ((s \otimes t) \cdot j)
  = \partial_\alpha(s \otimes t) \cdot j + (s \otimes t) \cdot (\alpha
  \otimes 1 - 1 \otimes \alpha) \cdot j
  \]
  in $I^\bullet$, for all $s,t\in S$ and $j\in I^\bullet$, the
  following holds
  \begin{equation}
    \label{eq:68}
    \mathrm{ev}(\alpha \cdot (\varphi \otimes i)) = \alpha \cdot \mathrm{ev}(\varphi \otimes i)\,.
  \end{equation}

  $\Hom_{S^e}(P^\bullet,S^e)$ is also a bounded complex of
  projective right $S^e$-modules. Hence, the functor
  $\Hom_{S^e}(P^\bullet,S^e) \otimes_{S^e}-$ preserves
  quasi-isomorphisms. Thus, $j \colon U^e\to
  I^\bullet$ induces a quasi-isomorphism of right $U^e$-modules,
  \begin{equation}
    \label{eq:38}
    \mathrm{Id} \otimes j \colon \Hom_{S^e}(P^\bullet,S^e) \otimes_{S^e}
    U^e \to \Hom_{S^e}(P^\bullet, S^e) \otimes_{S^e} I^\bullet\,.
  \end{equation}
  Endow the left-hand side term with the following action of $L$.
  For all $\alpha\in L$, $\varphi \in \Hom_{S^e}(P^\bullet,S^e)$ and
  $u,v\in U$, denote by $\alpha \cdot (\varphi \otimes (u \otimes v))$
  the following (well-defined) element of $\Hom_{S^e}(P^\bullet,S^e)
  \otimes_{S^e} U^e$:
  \[
  \alpha \cdot \varphi \otimes (u \otimes v)
  + \varphi \otimes ((\alpha \otimes 1 - 1 \otimes \alpha) \cdot (u
  \otimes v))\,;
  \]
  The assignment $\varphi \otimes (u \otimes v) \mapsto
  \alpha  \cdot (\varphi \otimes (u \otimes v))$ is a morphism of
  complexes of $R$-modules from $\Hom_{S^e}(P^\bullet,S^e) \otimes_{S^e}
  U^e$ to itself, and
  \[
  (\mathrm{Id}\otimes j)( \alpha \cdot (\varphi \otimes (u \otimes v))
  = \alpha \cdot ((\mathrm{Id} \otimes j)(\varphi \otimes (u \otimes
  v)))
  \]
  because $j\colon U^e \to I^\bullet$ is a morphism of complexes of
  $U^e-U^e$-bimodules.

  Since $U^e$ is projective in $\Mod(S^e)$, there is
  an isomorphism of right $U^e$-modules,
  \begin{equation}
    \label{eq:27}
    H^\bullet(\Hom_{S^e}(P^*,S^e) \otimes_{S^e} U^e)
    \simeq \Ext_{S^e}^\bullet(S,S^e) \otimes_{S^e} U^e\,.
  \end{equation}
  For all cocycles $\varphi \in \Hom_{S^e}(P^\bullet,S^e)$, with
  cohomology class denoted by $c$, and for all $\alpha \in L$ and
  $u,v\in U$, the image under (\ref{eq:27}) of the cohomology class of
  \[
  \alpha \cdot (\varphi \otimes (u \otimes v))
  \]
  is
  \begin{equation}
    \label{eq:28}
    \alpha \cdot c \otimes (u \otimes v) + c \otimes ((\alpha \otimes 1
    - 1 \otimes \alpha) \cdot (u \otimes v))
  \end{equation}
  where $\alpha \cdot c$ is defined in \ref{sec:acti-l-extb} (see (\ref{eq:20}))

  Combining \eqref{eq:36}, \eqref{eq:37}, \eqref{eq:38}, \eqref{eq:27}
  yields an isomorphism of right $U^e$-modules,
  \begin{equation}
    \label{eq:29}
    \Ext_{S^e}^\bullet(S,S^e) \otimes_{S^e} U^e \xrightarrow{\sim}
    H^\bullet(G(I^*))\,,
  \end{equation}
  such that, for all $\alpha \in L$, $c\in \Ext_{S^e}^\bullet(S,S^e)$
  and $u,v\in U$, if $\gamma$ denotes the image of $c \otimes (u \otimes
  v)$ under (\ref{eq:29}), then $\alpha \cdot \gamma$ is the image of
  (\ref{eq:28}).

  Thus, applying part (1) of Lemma~\ref{sec:1} to
  $N=\Ext_{S^e}^\bullet(S,S^e)$ yields
  the announced conclusion.
\end{proof}

\begin{proof}[Proof of Proposition~\ref{sec:inverse-dual-bimod}]
  The statement relative to the left $U$-module structure on
  $\Lambda^d_SL^\vee \otimes \Ext^\bullet_{S^e}(S,S^e)$ follows from
  Lemma~\ref{sec:left-right-u}, and Lemma~\ref{sec:u-homol-smooth}
  shows that $U$ is homologically smooth.  The (first quadrant,
  cohomological) spectral sequence of the bicomplex
  \begin{equation}
    \label{eq:31}
    ( \Hom_U(Q^p,U) \otimes_U G(I^q) )_{p,q}
  \end{equation}
  converges to $H^\bullet(\Hom_U(Q^*,U) \otimes_U G(I^*))$ and its
  $E_2^{p,q}$-term is, for all $p,q\in \mathbb Z$,
  \[
  H^p_h( H^q_v( \Hom_U(Q^\bullet,U) \otimes_U G(I^\bullet) )\,.
  \]
  Since $\Hom_U(Q^\bullet,U)$ consists of
  projective right $U$-modules, there is an isomorphism of right
  $U^e$-modules, for all $p,q\in \mathbb Z$,
  \begin{equation}
    \label{eq:33}
    H^q(\Hom_U(Q^p,U) \otimes_U G(I^\bullet))
    \simeq
    \Hom_U(Q^p,U) \otimes_U H^q(G(I^\bullet))\,.
  \end{equation}
  The description of $H^\bullet(G(I^*))$ made in
  Lemma~\ref{sec:cohomology-gibullet} combines with (\ref{eq:33}) into
  the following isomorphism of right $U^e$-modules, for all
  $p,q \in \mathbb Z$,
  \begin{equation}
    \label{eq:39}
    H^q(\Hom_U(Q^p,U) \otimes_U G(I^\bullet)) \simeq \Hom_U(Q^p,U)
    \otimes_U (\mathrm{Ext}_{S^e}^q(S,S^e) \otimes_{S^e} U^e)\,.
  \end{equation}
  Using Lemma~\ref{sec:1} (part (2)), this isomorphism yields
  isomorphism of right $U^e$-modules, for all $p,q \in \mathbb Z$:
  \begin{equation}
    \label{eq:34}
    H^q(\Hom_U(Q^p,U) \otimes_U G(I^\bullet))
    \simeq
    F( \Hom_U(Q^p,U) \otimes_S \Ext_{S^e}^q(S,S^e) )\,.
  \end{equation}
  Given that $F$ is an exact functor, that $\Ext^q_{S^e}(S,S^e)$ is
  projective in $\Mod(S)$ for all $q$ and that $(S,L)$ has duality in
  dimension $d$, it follows from (\ref{eq:34}) that there is an
  isomorphism of right $U^e$-modules, for all $p,q \in \mathbb Z$,
  \[
  H^p_h ( H^q_v ( \Hom_U(Q^\bullet,U) \otimes_U G(I^\bullet) ) )
  \simeq
  \left\{
    \begin{array}{rl}
      F( \Ext_U^d(S,U) \otimes_S \Ext^q_{S^e}(S,S^e))
      &
        \text{if $p=d$,} \\
      0
      &
        \text{if $p \neq d$.}
    \end{array}
  \right.
  \]
  Therefore, the spectral sequence of the bicomplex (\ref{eq:31})
  degenerates at $E_2$. Thus,
  \[
  H^\bullet(\Hom_U(Q^*,U) \otimes_U G(I^*)) \simeq
  F(\Ext_U^d(S,U) \otimes_S \Ext^{\bullet-d}_{S^e}(S,S^e))\ \ \text{in
    $\Mod(S^e)$.}
  \]
  The conclusion follows from (\ref{eq:26}) and from the
  isomorphism $\Ext^d_U(S,U)\simeq \Lambda^d_SL^\vee$ in $\Mod(U)$
  established in \cite[Theorem 2.10]{MR1696093}
\end{proof}

\subsection{Proof of the main theorem}
\label{sec:proof-theor-refs}

\begin{proof}[Proof of Theorem~\ref{sec:introduction}]
  Following Proposition~\ref{sec:inverse-dual-bimod}, $U$ is
  homologically smooth and there is an isomorphism of graded right
  $U^e$-modules
  \[
  \Ext^\bullet_{U^e}(U,U^e) \simeq F(\Lambda^d_SL^\vee \otimes_S
  \Ext^{\bullet -d}_{S^e}(S,S^e))\,.
  \]
  According to Proposition~\ref{sec:4}, the functor $F$ transforms left
  $U$-modules that are invertible as $S$-modules into invertible
  $U$-bimodules. Note that 
  \begin{itemize}
  \item $\Lambda^d_SL^\vee$ is invertible as an $S$-module because $L$
    is projective with constant rank and
  \item $\Ext^\bullet_{S^e}(S,S^e)$ is concentrated in degree $n$ and
    $\Ext^n_{S^e}(S,S^e)$ is invertible as an $S$-module because $S$ has
    Van den Bergh duality.
  \end{itemize}
  Thus, $\Ext^\bullet_{U^e}(U,U^e)$ is concentrated in degree $n+d$ and
  $\Ext^{n+d}_{U^e}(U,U^e)$ is invertible as a $U$-bimodule. Hence, $U$
  has Van den Bergh duality in dimension $n+d$.
\end{proof}

\subsection{Proof of Theorem~\ref{sec:introduction-1}}
\label{sec:proof-theor-refs-1}

The hypotheses of Theorem~\ref{sec:introduction-1} are
assumed throughout the subsection. Let $\varphi_L$ be a free generator of
the $S$-module $\Lambda^d_SL^\vee$. Let $e_S$ be a free generator of the
$S$-module $\Ext^n_{S^e}(S,S^e)$. Therefore, there exist mappings
\[
\lambda_L,\lambda_S \colon L \to S
\]
such that, for all $\alpha \in L$,
\[
\left\{
  \begin{array}{rcl}
    \alpha \cdot e_S & = & \lambda_S(\alpha)\cdot e_S \\
    \varphi_L \cdot \alpha & = & \lambda_L(\alpha)\cdot \varphi_S\,.
  \end{array}
\right.
\]
Some basic properties of these are summarised below.
\begin{lem}
  \label{sec:trace-maps-lambda_l}
  Let $\lambda$ be either one of $\lambda_S$ or $\lambda_L$. Then, for
  all $\alpha,\beta\in L$ and $s\in S$,
  \begin{enumerate}
  \item $\lambda(s\alpha) = s \lambda(\alpha) - \alpha(s)$;
  \item $\lambda([\alpha,\beta]) = \alpha(\lambda(\beta)) - \beta(\lambda(\alpha))$.
  \end{enumerate}
\end{lem}
\begin{proof}
  Assume that $\lambda = \lambda_S$. Let $s\in S$ and $\alpha \in
  L$. Then, using \ref{sec:acti-l-extb},
  \[
  \begin{array}{rcl}
    (s \alpha) \cdot e_S
    & = &
          s \cdot(\alpha \cdot e_S) -\alpha(s) \cdot e_S \\
    & = &
          (s \lambda(\alpha) - \alpha(s)) \cdot e_S,
  \end{array}
  \]
  which proves (1), and
  \[
  \begin{array}{rcl}
    \alpha \cdot (\beta \cdot e_S)
    & = &
          \alpha \cdot (\lambda(\beta)\cdot e_S) \\
    & = &
          \alpha(\lambda(\beta)) \cdot e_S + \lambda(\beta) \cdot
          (\alpha \cdot e_S) \\
    & = &
          (\alpha(\lambda(\beta)) + \lambda(\alpha) \lambda(\beta))
          \cdot e_S,
  \end{array}
  \]
  from which (2) may be proved directly. The proof when
  $\lambda=\lambda_L$ is analogous, using the right $U$-module
  structure of $\Lambda^d_SL^\vee$ instead of \ref{sec:acti-l-extb}.
\end{proof}

As proved later, the following automorphism is a Nakayama automorphism
for $U$.
\begin{lem}
  \label{sec:an-automorphism-u}
  There exists a unique $R$-algebra homomorphism,
  \[
  \nu \colon U \to U\,,
  \]
  such that, for all $s\in S$ and $\alpha \in L$,
  \[
  \left\{
    \begin{array}{rcl}
      \nu(s) & = & s \\
      \nu(\alpha) & = & \alpha + \lambda_L(\alpha) -
                        \lambda_S(\alpha)\,. 
    \end{array}
  \right.
  \]
  This is an automorphism of $R$-algebra.
\end{lem}
\begin{proof}
  The uniqueness is immediate. For all $\alpha \in L$, denote
  $\alpha + \lambda_L(\alpha) - \lambda_S(\alpha)$ by
  $\nu_\alpha$. Then, for all $s\in S$ and $\alpha,\beta\in L$,
  \[
  \begin{array}{rcl}
    [\nu_\alpha,\nu_\beta]
    & = &
          [\alpha + \lambda_L(\alpha) - \lambda_S(\alpha),
          \beta + \lambda_L(\beta) - \lambda_S(\beta)] \\
    & \underset{\text{\tiny Lemma \ref{sec:trace-maps-lambda_l}}}=
        &
          [\alpha,\beta] + \lambda_L([\alpha,\beta]) -
          \lambda_S([\alpha,\beta])
          =
          \nu_{[\alpha,\beta]} \\ \\
    \nu_{s \alpha}
    & = &
          s\alpha + \lambda_L(s\alpha) - \lambda_S(s\alpha) \\
    & \underset{\text{\tiny Lemma \ref{sec:trace-maps-lambda_l}}}=
        &
          s \alpha + s \lambda_L(\alpha) - s \lambda_S(\alpha)
    =
          s \nu_\alpha \\ \\
    {[\nu_\alpha, s]}
    & = &
          [\alpha +\lambda_L(\alpha) - \lambda_L(\alpha) , s] 
    =
          \alpha(s)\,.          
  \end{array}
  \]
  This proves the existence of $\nu$. Note that $\nu$ preserves the
  filtration of $U$ by the powers of $L$ and that $\mathrm{gr}(\nu)$ is
  the identity mapping of $U$. Accordingly, $\nu$ is bijective.
\end{proof}

Now it is possible to prove Theorem~\ref{sec:introduction-1}.
\begin{proof}[Proof of Theorem~\ref{sec:introduction-1}]
  From Theorem~\ref{sec:introduction}, $U$ has Van den Bergh
  duality in dimension $n+d$ and there is an isomorphism of
  $U$-bimodules
  \begin{equation}
    \label{eq:42}
    \Ext^{n+d}_{U^e}(U,U^e) \simeq F( \Lambda_S^d\Lambda^\vee \otimes_S
    \Ext^n_{S^e}(S,S^e)) \,,
  \end{equation}
  where the tensor product inside $F(\bullet)$ is a left $U$-module by
  (\ref{eq:50}).

  Recall that $\Lambda^d_SL^\vee$ and
  $\Ext^n_{S^e}(S,S^e)$ are freely generated by $\varphi_L$ and
  $e_S$, respectively. Therefore, the following mapping is an
  isomorphism of left $U$-modules (see Section~\ref{sec:functor-fcolon-modu})
  \begin{equation}
    \label{eq:43}
    \begin{array}{crcl}
      \Phi \colon & U & \to & F(\Lambda^d_SL^\vee \otimes_S
                              \Ext^n_{S^e}(S,S^e)) \\
                  & u & \mapsto & u \otimes (\varphi_L \otimes e_S)\,.
    \end{array}
  \end{equation}
  For all $s\in S$, $\alpha \in L$ and $u\in U$,
  \[
    \Phi(u) s
    =
          (u \otimes (\varphi_L \otimes e_S))\cdot s
    =
          us \otimes (\varphi_L \otimes e_S) 
    =
    \Phi(us)\,,
    \]
    \[
    \begin{array}{rcl}
    \Phi(u) \alpha
    & = &
          (u \otimes (\varphi_L \otimes e_S)) \cdot \alpha \\
    & = &
          u \alpha \otimes ( \varphi_L \otimes e_S )
          - u \otimes \alpha \cdot ( \varphi_L \otimes e_S ) \\
    & = &
          u \alpha \otimes ( \varphi_L \otimes e_S)
          - ( - u \otimes ( \varphi_L \cdot \alpha \otimes e_S )
          + u \otimes (\varphi_L \otimes \alpha \cdot e_S )) \\
    & = &
          ( u (\alpha + \lambda_L(\alpha) - \lambda_S(\alpha)))
          \otimes ( \varphi_L \otimes e_S ) \\
    & = &
          \Phi(u (\alpha + \lambda_L(\alpha) - \lambda_S(\alpha)))\,.
  \end{array}
  \]
  Thus, denoting by $\nu$ the automorphism of $U$ considered in
  Lemma~\ref{sec:an-automorphism-u}, then, for all $u,v\in U$,
  \begin{equation}
    \label{eq:44}
    \Phi(u)\cdot v = \Phi(u \nu(v))\,.
  \end{equation}
  Combining (\ref{eq:42}), (\ref{eq:43}) and
  (\ref{eq:44}) yields that there is an isomorphism of bimodules
  \[
  \Ext^{n+d}_{U^e}(U,U^e) \simeq U^\nu\,.
  \]
  Since $\lambda_S=-\mathrm{div}$ (see
  Proposition~\ref{sec:particular-case-van}), this proves
  Theorem~\ref{sec:introduction-1}.
\end{proof}

\subsection{Case of Poisson algebras}
\label{sec:case-poiss-algebr}

\begin{proof}[Proof of Corollary~\ref{sec:main-results-1}]
  From Proposition~\ref{sec:relat-regul}, $S$ has Van den Bergh
  duality in dimension $n$. Moreover,
  Proposition~\ref{sec:particular-case-van} yields an isomorphism of
  $S$-modules
  $\Lambda^n_S\mathrm{Der}_R(S) \simeq \Ext^n_{S^e}(S,S^e)$ which is
  compatible with the action of Lie derivatives. Finally, according to
  (\ref{eq:41}), the dualising module of $(S,\Omega_{S/R})$ is
  $\Lambda^n_S\mathrm{Der}_R(S)$ with right $U$-module structure such
  that, for all $s\in S$ and
  $\varphi \in \Lambda^n_S\mathrm{Der}_R(S)$,
  \[
  \varphi \cdot ds = -\mathcal L_{\{s,-\}}(\varphi)\,.
  \]
  Using these considerations, the corollary follows from
  Theorems~\ref{sec:introduction} and \ref{sec:introduction-1}.
\end{proof}

\section{Examples}
\label{sec:example}

\subsection{The case where $L$ is free as an $S$-module}
\label{sec:when-l-free}

In this subsection, it is assumed that $L$ is free as an
$S$-module. Consider a basis $(\alpha_1,\ldots,\alpha_d)$ of $L$ over
$S$. Denote the dual basis of $L^\vee$ by
$(\alpha_1^*,\ldots,\alpha_d^*)$. In particular, $\Lambda^d_SL^\vee$ is
free of rank one in $\Mod(S)$, with a generator denoted by
$\varphi_L$,
\[
\varphi_L = \alpha_1^*\wedge \cdots \wedge \alpha_d^*\,.
\]
For all $i\in \{1,\ldots,d\}$, consider the matrix of $\mathrm{ad}_{\alpha_i}$, denoted by $(s_{j,k}^i)_{j,k} \in M_d(S)$. Hence, for all
$i,k\in \{1,\ldots,d\}$,
\[
[\alpha_i,\alpha_k] = \sum_{j=1}^d s^i_{j,k}\alpha_j\,.
\]
In this situation, the action of $L$ on $\Lambda^\bullet_SL$ by Lie
derivatives specialises as follows. For all $i,j,k\in
\{1,\ldots,d\}$,
\[
\begin{array}{rcl}
  (\lambda_{\alpha_i}(\alpha_j^*))(\alpha_k)
  & = &
        \alpha_i(\alpha_j^*(\alpha_k)) -
        \alpha_j^*([\alpha_i,\alpha_k]) \\
  & = &
        - s_{j,k}^i\,.
\end{array}
\]
Hence, for all $i,j\in \{1,\ldots,d\}$,
\[
\lambda_{\alpha_i}(\alpha_j^*) = - \sum_{k=1}^d s_{j,k}^i
\alpha_k^*\,.
\]
Thus, the right $U$-module structure of $\Lambda^d_SL^\vee$ is such
that, for all $\alpha\in L$,
\begin{equation}
  \label{eq:69}
  \varphi_L \cdot \alpha = \mathrm{Tr}(\mathrm{ad}_{\alpha}) \varphi_L\,.
\end{equation}
Using this simplified description of $\Lambda^d_SL^\vee$ yields the
following corollary of the main theorems of this article. 
\begin{corr}
  \label{sec:when-l-free-1}
  Let $R$ be a commutative ring. Let $(S,L)$ be a Lie-Rinehart algebra
  of $R$. Denote by $U$ its enveloping algebra. Assume that
  \begin{itemize}
  \item $S$ is flat as an $R$-module,
  \item $S$ has Van den Bergh duality in dimension $n$,
  \item $L$ is free of rank $d$ as an $S$-module.
  \end{itemize}
  Let $(\alpha_1,\ldots,\alpha_d)$ be a basis of $L$ over $S$ as
  considered previously. Then, $U$ has Van den Bergh duality in
  dimension $n+d$ and there is an isomorphism of $U$-bimodules
  \[
  \Ext^{n+d}_{U^e}(U,U^e) \simeq U \otimes_S \Ext^n_{S^e}(S,S^e)\,,
  \]
  where the left $U$-module structure on $U \otimes_S
  \Ext^n_{S^e}(S,S^e)$ is the natural one and the right
  $U$-module structure is such that, for all $u\in U$, $e\in
  \Ext^n_{S^e}(S,S^e)$ and $\alpha \in L$,
  \[
  (u \otimes e) \cdot \alpha = u\alpha \otimes e + u \otimes \mathrm{Tr}(\mathrm{ad}_{\alpha})e - u \otimes \mathcal
  L_{\partial_\alpha}(e)\,.
  \]
  If, moreover, $S$ is Calabi-Yau, then $U$ is skew Calabi-Yau and
  each volume form on $S$ determines a Nakayama automorphism $\nu \in
  \mathrm{Aut}_R(U)$ such that, for all $s\in S$ and $\alpha \in L$,
  \[
  \left\{
    \begin{array}{rcl}
      \nu(s) & = & s \\
      \nu(\alpha) & = & \alpha + \mathrm{Tr}(\mathrm{ad}_{\alpha}) + \mathrm{div}(\partial_\alpha)\,,
    \end{array}\right.
  \]
  where $\mathrm{div}$ denotes the divergence of the chosen volume form.  
\end{corr}
\begin{proof}
  In view of \eqref{eq:69}, there is an isomorphism of right
  $U$-modules
  \[
  \Lambda^d_SL^\vee \simeq S\,,
  \]
  where the right $U$-module structure on the right-hand side term is
  such that, for all $\alpha\in L$,
  \[
  1\cdot \alpha = \mathrm{Tr}(\mathrm{ad}_\alpha)\,.
  \]
  The corollary therefore follows directly from
  Theorems~\ref{sec:introduction} and \ref{sec:introduction-1}.
\end{proof}

The previous corollary applies to any Lie-Rinehart algebra arising
from a Poisson structure on $R[x_1,\ldots,x_n]$,
$n\in \mathbb N\backslash\{0,1\}$. 
\begin{exple}
  Let $S=R[x,y]$. Let $P\in S$. This defines a Poisson
  structure on $S$ such that
  \[
  \{x,y\} = P \,.
  \]
  Let $L:= \Omega_{S/R}$ and consider $(S,L)$ as a Lie-Rinehart
  algebra over $R$ such that, for all $s,t\in S$
  \begin{itemize}
  \item $[ds,dt] = d\{s,t\}$;
  \item $\partial_{ds} = \{s,-\}$.
  \end{itemize}
  Then $(dx,dy)$ is a basis of $\Omega_{S/R}$ over $S$. Note that
  \[
  \left\{
    \begin{array}{lclcl}
      \mathrm{Tr}(\mathrm{ad}_{dx})
      & = &
            \mathrm{div}(\partial_{dx})
      & = & \frac{\partial P}{\partial y} \\
      \mathrm{Tr}(\mathrm{ad}_{dy})
      & = &
            \mathrm{div}(\partial_{dy})
      & = & -\frac{\partial P}{\partial x}\,.
    \end{array}
    \right.
  \]
  From Corollary~\ref{sec:when-l-free-1}, $U$ is skew Calabi-Yau in
  dimension $4$ and has a Nakayama automorphism
  $\nu \in \mathrm{Aut}_R(S)$ such that
  \[
  \left\{
  \begin{array}{rclrcl}
    \nu(x) & = & x\,, & \nu(dx) & = & dx +2\frac{\partial P}{\partial
                                       y} \\ \\
    \nu(y) & = & y\,, & \nu(dy) & = & dy -2\frac{\partial P}{\partial
                                       x}\,.
  \end{array}\right.
\]
By considering the filtration of $U$ by the powers of the
image of $L$ in $U$, with associated graded algebra the symmetric algebra
of $L$ over $S$ (see \cite[Theorem 3.1]{MR0154906}), it appears that
$U^\times = S^\times = R^\times$.
Accordingly, $U$ has no nontrivial inner automorphism. Consequently, $U$
is Calabi-Yau if and only if $\nu=\mathrm{Id}_U$, that is, if and only if
$\mathrm{char}(R)=2$, or else $P\in R$.
\end{exple}

\begin{exple}
  Let $S=R[x,y,z]$. Let $P_x,P_y,P_z\in S$ be such that
  \[
  \overrightarrow P \wedge \mathrm{curl}(\overrightarrow P) = 0\,,
  \]
  where $\overrightarrow P$ denotes
  \[
  \left(
    \begin{smallmatrix}
      P_x \\ P_y \\ P_z
    \end{smallmatrix}\right)\,.
  \]
  Hence, the following defines a Poisson
  bracket on $S$,
  \[
  \{x,y\} = P_z\,,\ \{y,z\} = P_x\,,\ \{z,x\} = P_y\,.
  \]
  As in the previous example, let $(S,L:=\Omega_{S/R})$ be the
  associated Lie-Rinehart algebra over $R$. As is well-known,
  \[
  \begin{array}{lll}
  \{x,-\} = P_z \frac{\partial}{\partial y} -
    P_y\frac{\partial}{\partial z}\,,
    &
  \{y,-\} = P_x \frac{\partial}{\partial z} -
    P_z\frac{\partial}{\partial x}\,,
    &
  \{z,-\} = P_y \frac{\partial}{\partial x} -
      P_x\frac{\partial}{\partial y}\,.
  \end{array}
  \]
  Therefore, using the basis $(dx,dy,dz)$ of $\Omega_{S/R}$ over $S$,
  \[
  \left(
    \begin{smallmatrix}
      \mathrm{div}(\partial_{dx}) \\
      \mathrm{div}(\partial_{dy}) \\
      \mathrm{div}(\partial_{dz})
    \end{smallmatrix}
  \right)
  = \left(
    \begin{smallmatrix}
      \mathrm{Tr}(\mathrm{ad}_{dx}) \\
      \mathrm{Tr}(\mathrm{ad}_{dy}) \\
      \mathrm{Tr}(\mathrm{ad}_{dz}) \\
    \end{smallmatrix}
  \right)
  = \mathrm{curl}(\overrightarrow{P})\,.
  \]
  Using Corollary~\ref{sec:when-l-free-1}, it follows that $U$ is
  skew Calabi-Yau in dimension $6$ and has a Nakayama automorphism
  $\nu \in \mathrm{Aut}_R(S)$ such that
  \[
  \begin{array}{lcl}
  \left(
    \begin{smallmatrix}
      \nu(x) \\ \nu(y) \\ \nu(z)
    \end{smallmatrix}
  \right)
  =
  \left(
  \begin{smallmatrix}
    x \\ y \\ z
  \end{smallmatrix}
    \right)
 &
   \text{and}
 &
  \left(
    \begin{smallmatrix}
      \nu(dx) \\ \nu(dy) \\ \nu(dz)
    \end{smallmatrix}
  \right)
  =
  \left(
  \begin{smallmatrix}
    dx \\ dy \\ dz
  \end{smallmatrix}
    \right) + 2\,\mathrm{curl}(\overrightarrow P)\,.
  \end{array}
  \]
  As in the previous example, there are no nontrivial inner
  automorphisms for $U$. Hence, $U$ is Calabi-Yau if and only if $\mathrm{char}(R) =2$, or else $\mathrm{curl}(\overrightarrow P) =0$. In
  particular, when $R$ contains $\mathbb Q$ as a subring, then $U$ is
  Calabi-Yau if and only if the Poisson bracket is Jacobian, that is,
  there exists $Q \in S$ such that $P=\overrightarrow{\mathrm{grad}}(Q)$.
\end{exple}

By the Quillen-Suslin Theorem, when $R$ is a field and $n\in \mathbb
N$, any $R[x_1,\ldots,x_n]$-module that is finitely generated and
projective is free. Hence, Corollary~\ref{sec:when-l-free-1} also
applies to all Lie-Rinehart algebras of the shape
$(R[x_1,\ldots,x_n],L)$, where $R$ is a field.

\subsection{On two dimensional Nambu-Poisson structures}
\label{sec:class-two-dimens}

Following Corollary~\ref{sec:main-results-1}, $U$ is skew Calabi-Yau
when $S$ is flat over $R$ and Calabi-Yau and $(S,L)$ is given by a
Poisson bracket on $S$. Assuming these properties, this section
computes a Nakayama automorphism of $U$ for a class of examples of two
dimensional Nambu-Poisson structures (see \cite[Section
8.3]{MR2906391}).

Let $S = R[x,y,z]/(P)$ where $P=1+T$ for some $T\in R[x,y,z]$ which is
$(p,q,r)$-homogeneous in the sense that $p,q,r\in R$ and
$t:=p\alpha +q\beta +r\gamma$ is a unit in $R$ which does not depend
on the monomial $x^\alpha y^\beta z^\gamma$ appearing in $T$.  The
hypotheses imply the following equality in $S$
\begin{equation}
  \label{eq:54}
  px\frac{\partial P}{\partial x} + qy \frac{\partial P}{\partial y} +
  rz \frac{\partial P}{\partial z} = -t\ (\in R^\times)\,.
\end{equation}
Let $Q\in R[x,y,z]$ and endow $S$ with the Poisson structure such that
\begin{equation}
  \label{eq:65}
  \{x,y\} = Q\frac{\partial P}{\partial z}\,,\ \{y,z\} = Q\frac{\partial
    P}{\partial x}\,,\ \{z,x\} = Q \frac{\partial P}{\partial y}\,.
\end{equation}
Consider $(S,L:=\Omega_{S/R})$ as a Lie-Rinehart algebra
such that, for all $s,t,s'\in S$,
\begin{itemize}
\item $[ds,dt] = d\{s,t\}$,
\item $(sdt)(s') = s\{t,s'\}$.
\end{itemize}

Consider the following $2$-form on $S$
\[
\omega_S = px dy\wedge dz + qy dz \wedge dx + rz dx \wedge dy\,.
\]
According to (\ref{eq:54}), $\Omega_{S/R}$ is a projective $S$-module
of rank $2$. And the relation
\[
\frac{\partial P}{\partial x}dx+ \frac{\partial P}{\partial y}dy+
\frac{\partial P}{\partial z}dz=0
\]
in $\Omega_{S/R}$ yields the following relations in
$\Lambda^2_S\Omega_{S/R}$
\[
\begin{array}{rcl}
  \frac{\partial P}{\partial x}\,dx\wedge dy & = & \frac{\partial
                                                   P}{\partial z}\,
                                                   dy\wedge dz \\ \\
  \frac{\partial P}{\partial y}\,dy\wedge dz & = & \frac{\partial
                                                   P}{\partial x}\,
                                                   dz\wedge dx \\ \\
  \frac{\partial P}{\partial z}\,dz\wedge dx & = & \frac{\partial
                                                   P}{\partial y}\,
                                                   dx\wedge dy\,.
\end{array}
\]
Combining with (\ref{eq:54}) yields
\[
\begin{array}{rcl}
  dx\wedge dy & = & -t^{-1}\frac{\partial P}{\partial z} \omega_S \\ \\
  dy\wedge dz & = & -t^{-1}\frac{\partial P}{\partial x} \omega_S \\ \\
  dz\wedge dx & = & -t^{-1}\frac{\partial P}{\partial y} \omega_S\,.
\end{array}
\]
Thus, $\omega_S$ is a volume form of $S$.

In order to determine the divergence of $\omega_S$, consider the
derivations $\delta_x,\delta_y,\delta_z\in \mathrm{Der}_R(S)$ given by
\[
\begin{array}{lll}
  \begin{array}{cccc}
    \delta_x \colon & x & \mapsto & 0 \\ \\
                    & y & \mapsto & \frac{\partial P}{\partial z} \\ \\
                    & z & \mapsto & -\frac{\partial P}{\partial y}
  \end{array}
                    &
                      \begin{array}{cccc}
                        \delta_y \colon & x & \mapsto &
                                                        -\frac{\partial P}{\partial z} \\ \\
                                        & y & \mapsto & 0 \\ \\
                                        & z & \mapsto & \frac{\partial P}{\partial x}
                      \end{array}
                    &
                      \begin{array}{cccc}
                        \delta_z \colon & x & \mapsto & \frac{\partial
                                                        P}{\partial y}
                        \\ \\
                                        & y & \mapsto &
                                                        -\frac{\partial P}{\partial x} \\ \\
                                        & z & \mapsto & 0\,.
                      \end{array}
\end{array}
\]
Note that
\begin{center}
  $\{x,-\} = Q\delta_x$, $\{y,-\} = Q\delta_y$ and $\{z,-\} = Q\delta_z$.
\end{center}
Then,
\[
\begin{array}{rcll}
  \iota_{\delta_x}(\omega_S)
  & = &
        \iota_{\delta_x}(px dy\wedge dz + qy dz \wedge dx + rz dx
        \wedge dy) \\ \\
  & = &
        px(\frac{\partial P}{\partial z}dz + \frac{\partial
        P}{\partial y}dy) - qy \frac{\partial P}{\partial y} dx
        - rz \frac{\partial P}{\partial z}dx \\ \\
  & = &
        tdx & \text{(see (\ref{eq:54})).}
\end{array}
\]
Therefore, using the symmetry between $x$, $y$ and $z$,
\[
\mathrm{div}(\delta_x) = \mathrm{div}(\delta_y) = \mathrm{div}(\delta_z)=0\,.
\]
Apply Lemma~\ref{sec:trace-maps-lambda_l} taking into account that
$\lambda_S=-\mathrm{div}$ (see \eqref{eq:56}; then,
\[
\mathrm{div}(\{x,-\}) = \mathrm{div}(Q\delta_x) = Q \mathrm{div}(\delta_x) +
\delta_x(Q)\,.
\]
Therefore,
\begin{equation}
  \label{eq:32}
  \mathrm{div}(\{x,-\}) = \frac{\partial Q}{\partial y}\frac{\partial
    P}{\partial z} - \frac{\partial Q}{\partial z}\frac{\partial
    P}{\partial y}\,.
\end{equation}
Applying Corollary~\ref{sec:main-results-1} gives that the
enveloping algebra $U$ of $(S,\Omega_{S/R})$ is skew Calabi-Yau. It
has a Nakayama automorphism $\nu \colon U\to U$ such that, for all
$s\in S$,
\[
\left\{
\begin{array}{rcl}
  \nu(s) & = & s \\
  \left(
  \begin{smallmatrix}
    \nu(dx) \\ \nu(dy) \\ \nu(dz)
  \end{smallmatrix}
  \right)
         & = &
               \left(
               \begin{smallmatrix}
                 dx \\ dy \\ dz
               \end{smallmatrix}
  \right)
  +
  2\,\overrightarrow{\mathrm{grad}}(Q) \wedge
  \overrightarrow{\mathrm{grad}}(P)\,.
\end{array}\right.
\]
\bibliographystyle{alpha}
\bibliography{biblio}
\end{document}